\documentclass[reqno]{amsart}
\usepackage[margin=1.5in]{geometry}

\usepackage{amsmath,amsthm,amscd,amssymb}
\usepackage{latexsym}
\usepackage{color}

\usepackage{bbm}

\usepackage{todonotes}
\usepackage[all]{xy}

\usepackage{calrsfs}
\DeclareMathAlphabet{\pazocal}{OMS}{zplm}{m}{n}

\numberwithin{equation}{section}

\newtheorem{Thm}{Theorem}[section]

\newtheorem{Lem}[Thm]{Lemma}
\newtheorem{Prop}[Thm]{Proposition}
\newtheorem{Def}[Thm]{Definition}

\newtheorem{Cond}[Thm]{Condition}

\newcommand{\e}{\varepsilon}

\newcommand{\eps}{\theta}
\newcommand{\epp}{\epsilon}

\renewcommand{\L}{\Lambda}

\newcommand{\N}{\mathbb{N}}
\newcommand{\Z}{\mathbb{Z}}
\renewcommand{\P}{\mathbb{P}}
\newcommand{\R}{\mathbb{R}}
\newcommand{\E}{\mathbb{E}}
\newcommand{\T}{\mathbb{T}}

\renewcommand{\L}{\Lambda}

\newcommand{\Mb}{\pazocal{M}}

\newcommand{\Pa}{\mathcal{P}}
\newcommand{\Qb}{\pazocal{Q}}
\newcommand{\Pam}{\mathcal{R}}
\newcommand{\Pb}{\pazocal{P}}

\newcommand{\id}{\mathbbm{1}}

\newcommand{\mf}[1]{{\mathfrak #1}}

\newcommand{\bb}[1]{{\mathbb #1}}

\newcommand{\dsp}{\displaystyle}

\title[Hydrodynamic limits in random environments] {On Hydrodynamic Limits in Sinai-type random environments}

\author{Claudio Landim}
\address{IMPA, Estrada Dona Castorina 110, J. Botanico, 22460 Rio de Janeiro, Brazil; CNRS UPRES A 6085, Universit\'e de Rouen, 76128 Mont Saint Aignan Cedex, France}
\email{landim@impa.br}
\author{Carlos G. Pacheco}
\address{Department of Mathematics,
CINVESTAV-IPN,
Av. IPN 2508, CP. 07360,
Mexico City, Mexico}
\email{cpacheco@math.cinevistav.mx}
\author{Sunder Sethuraman}
\address{Department of Mathematics, University of Arizona,  Tucson, AZ 85721, USA}
\email{sethuram@math.arizona.edu}
\author{Jianfei Xue}
\address{Department of Mathematics, University of Arizona,  Tucson, AZ 85721,USA}
\email{jxue@math.arizona.edu}

\begin{document}

\begin{abstract}
We investigate the hydrodynamical behavior of a system of random walks with zero-range interactions moving in a common `Sinai-type' random environment on a one dimensional torus.  The hydrodynamic equation found is a quasilinear SPDE with a `rough' random drift term coming from a scaling of the random environment and a homogenization of the particle interaction.  Part of the motivation for this work is to understand how the space-time limit of the particle mass relates to that of the known single particle Brox diffusion limit.    In this respect, given the hydrodynamic limit shown, we describe formal connections through a two scale limit.

\end{abstract}



\maketitle


\section{Introduction}

The purpose of this article is to understand the `quenched' hydrodynamical behavior of a system of random walks interacting via zero-range dynamics in a common Sinai-type random environment on $\Z$.  Our motivation is two-fold:  On the one hand,
since the single particle scaling limit in a Sinai-type random environment is a Brox diffusion, it is natural to investigate the micro to macro-behaviors in an interacting system of many particles.  On the other hand, although hydrodynamic limits have been studied with respect to a few interacting systems in a common random environment with traps, the limit here is different and of interest, namely a quasilinear SPDE driven in terms of a `rough' noise emerging from the random environment.  

\medskip
{\it `Sinai' random environments.}  A `Sinai' random environment on $\Z$ is a sequence of independent and identically distributed (i.i.d.) random variables $\{u_i\}_{i\in \Z}$, indexed over vertices, with the property that $c\leq u_0\leq 1-c$ for some constant $0<c<1/2$ and $E[\log (u_0/(1-u_0))] = 0$.  Define $\sigma^2= E[(\log (u_0/(1-u_0)))^2]$.  Let $U_n$ be the position of a discrete-time random walk in this random environment (RWRE):  
$$P\big(U_{n+1} = U_n +1| U_n, \{u_i\}\big) = 1-P\big(U_{n+1}=U_n -1|U_n, \{u_i\}\big) = u_{U_n}$$
 for $n\geq 1$ and $U_0=0$. 
When $\sigma^2>0$, Sinai \cite{Sinai} showed that $\sigma^2 U_n/(\log(n))^2$ converges weakly to a non-trivial random variable $U_\infty$, whose law was identified in Kesten \cite{Kesten} and Golosov \cite{Golosov}.  A functional limit theorem to a non-trivial process is problematic however as $\sigma^2 U_{\lfloor nt\rfloor}/(\log(n))^2 \Rightarrow U_\infty$, the limit process here being constant in time $t$.  

However, a continuous analog $X_t$ of the Sinai RWRE on $\R$ was introduced in Brox \cite{Brox}:  Formally, 
$$dX_t = dB_t -\frac{1}{2}W'(X_t)dt$$ and $X_0=0$ or that the generator $\mathcal{L}_{Brox}$ takes form 
$$\mathcal{L}_{Brox} = \frac{1}{2}e^{W(x)}\frac{d}{dx} (e^{-W(x)}\frac{d}{dx}).$$  Here, $B$ is a standard Brownian motion and $W$ is a two-sided Brownian motion on $\R$: $W(0)=0$, $W(x) = \sigma W_+(x)$ for $x>0$, and $W(x)=\sigma W_-(-x)$ for $x<0$, where $W_\pm$ are independent standard Brownian motions.  This description is only short hand, as $W$ is not differentiable a.s.   More carefully, the Brox diffusion is defined in terms of speed and scale measures:  
$$X_t = A^{-1}(B_{T^{-1}(t)}) \ \ {\rm where \ \ }A(y) = \int_0^y e^{W(z)}dz \ \ {\rm  and  \ \ }T(t) = \int_0^t e^{-2W(A^{-1}(B_s))}ds$$ for $y\in \R$ and $t\geq 0$.
Brox \cite{Brox} showed that $X_t/(\log t)^2\Rightarrow U_\infty$, the same limit as for the discrete Sinai RWRE convergence.

To connect the two models, Sinai-type random environments $\{u^{N}_i\}_{i\in \Z}$, in terms of a scaling parameter $N$, were introduced in Seignourel \cite{Seignourel}.  An example that we consider is that $\{u^{N}_i\}$ are i.i.d. over $i$ and $N$, and $u^{N}_i = 1/2 + r_i/\sqrt{N}$ where $\{r_i\}_{i\in \Z}$ is an i.i.d. sequence of random variables which are mean-zero and with finite variance.  Let $U^{N}_n$ be the corresponding RWRE with respect to the scaled enviornment $\{u^{N}_i\}$.  Seignourel \cite{Seignourel} showed that $\{N^{-1} U^{N}_{\lfloor N^2t\rfloor}: t\geq 0\}$ converges weakly to the Brox diffusion $\{X_t: t\geq 0\}$.  See also Andriopolous \cite{And} and Pacheco \cite{Pacheco} for extensions and variations of this convergence.

Of course, both Sinai's random walk and Brox diffusion are well studied objects with current developments.  As a partial list, we refer to books/surveys \cite{Zeitouni1}, \cite{Zeitouni2}, and recent works \cite{And},  \cite{Hu_Le_Mytnik}, \cite{MRS} and references therein for more discussion.

\medskip

{\it Hydrodynamics in random environments.} With respect to systems of many continuous-time random walks in a common random environment of different types, a `quenched' hydrodynamic limit (HDL) for the bulk space time mass density of the walks has been shown in some cases.  When `averaging' is possible with respect to the random environment, deterministic HDL's have been shown in some models.   For exclusion process with random conductances, see Faggionato \cite{F}, \cite{F2}, Jara and Landim \cite{JL}, and Nagy \cite{Nagy}.   For independent random walks in a ballistic random environment, see Peterson \cite{Peterson}.  For symmetric zero-range process, see Gon\c{c}alves and Jara \cite{GJ}.

When the random environment does not allow `averaging', a `quenched' HDL may involve random terms.   In Jara, Landim, Teixera \cite{JLT}, HDL is shown for a system of symmetric independent random walks in a common scaled `trap' environment on a torus, that one of the particles at a site $i$ jumps at rate $\xi^{N}_i$ to a nearest neighbor, where $\xi^{N}_i$ is random and heavy-tailed; here the limit equation involves a heavy-tailed subordinator arising from the random environment.  In Faggionato, Jara, Landim \cite{FJL}, HDL is shown for symmetric simple exclusion processes on a torus, with heavy-tailed random conductances on the bonds, which also involves a heavy-tailed subordinator coming from the random environment.  In Jara and Peterson \cite{JP}, HDL is shown for independent random walks in a random environment on $\Z$, where a single particle is transient but not ballistic, which incorporates a random term arising from the environment.

For a more general discussion on hydrodynamics of stochastic interacting particle systems, we refer as a partial list to books \cite{DP}, \cite{KL} and references therein.

\medskip
{\it Summary of results.} In these contexts, we consider a zero-range interacting particle system of random walks moving in a Sinai-type random environment $\{u^{N}_i = 1/2 + r_i/\sqrt{N}\}$ on a torus $\Z\setminus N\Z$.  Informally, in the zero-range process, a particle at site $i$, with $k$ particles, will jump with rate $g(k)/k$ and move to a neighbor $j=i\pm 1$ with probability $p(i,j)$. We point out the case $g(k)\equiv k$ is when the random walks are all independent.  When a random environment is imposed, $p(i,j) = u^{N}_i$ when $j=i+1$ and $= 1-u^{N}_i$ when $j=i-1$.  Since a Sinai random walker experiences many traps and moves slowly, one expects that the hydrodynamic limit to involve `drift' terms reflecting the environment.  

In principle, since infinitesimally the number of particles $\eta_i(t)$ at site $i$ at time $t$ varies according the generator action,
\begin{align*}
&L\eta_i(t) = \frac{1}{2}\big\{ g(\eta_{i+1}(t)) 
+ g(\eta_{i-1}(t)) - 2g(\eta_i(t))\big\} \\
&\ \ \ \ \ \ \ \ \ \ \ \ \ + \big\{g(\eta_{i-1}(t))(r_{i-1}/\sqrt{N}) - g(\eta_{i+1}(t))(r_{i+1}/\sqrt{N})\big\}dt,
\end{align*}
and $r_i/\sqrt{N} = [S(i) - S(i-1)]/\sqrt{N}$ where $S(\cdot)$ is the partial sum of $\{r_i\}$,
 one might expect the equation, scaling space and time diffusively, as
\begin{equation}
\label{expected_HDL}
\partial_t \rho = \frac{1}{2}\Delta \Phi(\rho) - 2 \nabla \big(W'\Phi(\rho)\big).
\end{equation}
Here, $\Phi$ is a homogenized version of the jump rate $g$.
Of course, since $W$ is not differentiable, the spatial white noise $W'$ has to be interpreted as a distribution, and the equation is ill-posed.  But, in a sense, an equation like this should represent the hydrodynamic equation, the term $W'$ reflecting the random environment, even after some averaging in the scaling limit has been taken.

In this article, as a way to obtain formally equation \eqref{expected_HDL}, we introduce a two-scale approach:  We will consider more regular random environments, those which average $\{u^{N}_i\}$ in small macroscopic blocks.  Namely, we consider $\{\bar u^{N}_i = 1/2 + q_i^N/\sqrt{N}\}$ where $q_i^N = (2\varepsilon N + 1)^{-1}\sum_{|j-i|\leq \varepsilon N} r_j$ and $\varepsilon>0$ is fixed.  Then, under a fixed random environment given by $\{\bar u^{N}_i\}$, that is in a `quenched' environment, we obtain in Theorem \ref{main thm} the hydrodynamic equation
\begin{equation}
\label{main_HDL}
\partial_t \rho^{(\varepsilon)}(t,x) = \frac{1}{2}\Delta \Phi \big(\rho^{(\varepsilon)}(t,x)) - 2 \nabla \big( \frac{W(x+\varepsilon)-W(x-\varepsilon)}{2\varepsilon}\Phi(\rho^{(\varepsilon)}(t,x))\big).
\end{equation}
We remark, in the $\varepsilon\downarrow 0$ limit of \eqref{main_HDL}, formally, one sees that some form of \eqref{expected_HDL} should emerge. 
Indeed, in Funaki et.~al.~\cite{FHSX}, the behavior of the limit $\rho^{(\varepsilon)}$ as $\varepsilon \downarrow 0$ is considered carefully and shown to solve \eqref{expected_HDL} in the `paracontrolled' sense.  See \cite{FHSX} for a general discussion about equations \eqref{expected_HDL} and \eqref{main_HDL} in connection to literature.

The proof of Theorem \ref{main thm} broadly employs the `entropy' method of Guo-Papanicoloau-Varadhan (GPV) (cf. Kipnis and Landim \cite{KL}), by applying an Ito's formula to the empirical measure of particle mass with respect to the zero-range evolution.  However, the random environment is not homogeneous and in particular is not translation-invariant or smooth, a key feature of the `GPV' technique to homogenize resulting nonlinear terms of the process in a replacement scheme.  The main technical work is to introduce `local' averages to piece together the `global' average in the hydrodynamic limit.  

To derive the homogenization, we make use of some averaging in time, afforded by spectral gap or mixing estimates of localized processes, to perform `local' $1$ and $2$-block replacements, leading to more `global' replacements.  These `local' replacements are not so broadly known, although related notions were used in Jara, Landim, and Sethuraman \cite{JLS}, \cite{JLS2}, and Fatkullin, Sethuraman, Xue \cite{FSX} to analyze tagged particle motion and Young diagram evolutions.  However, differently, in our context,
estimates on the random environment play a significant role in making the `local' replacements work.  We remark that in this work we do not assume the process rate $g$ is an increasing function, as it is in \cite{JLS2}, \cite{FSX}, that is that the process is `attractive', a technical condition which would allow use of `basic' particle couplings.

The limiting equation \eqref{main_HDL} is not a standard one as the factor $W'_\varepsilon:=(W(x+\varepsilon) -W(x-\varepsilon))/(2\varepsilon)$ is not smooth but in class $C^{1/2-\epsilon}$ for all $0<\epsilon<1/2$.  To finish the proof of the hydrodynamic limit in Theorem \ref{main thm}, we need to show uniqueness of weak solutions of \eqref{main_HDL}.  We can derive a certain continuum energy estimate by considering the microscopic particle system, namely that $\partial_x \Phi(\rho(t,x))$ can be defined in a weak sense.  Uniqueness of weak solutions to \eqref{main_HDL} in an appropriate class is then shown by a self-contained argument, which might be of separate interest.

Finally, to connect back to Brox diffusion, one might follow a tagged particle in this system.  Let $x^{N}(t)$ be the position at time $t$ of a tagged particle initially at the origin.  Then, from Ito's formula, 
$$x^{N}(t) = \frac{2}{\sqrt{N}}\int_0^t \bar u^{N}_{x^{N}(s)} \frac{g(\eta_s(x^{N}(s)))}{\eta_s(x^{N}(s))}ds + M^{N}(t).$$ 
One can compute that the quadratic variation of $M^{N}(t) = \int_0^t\frac{g(\eta_s(x^{N}(s)))}{\eta_s(x^{N}(s))}ds$.  Formally, given the hydrodynamic limit \eqref{main_HDL} and following the scheme in \cite{JLS},
one might expect limit points $v_t$ of $v^{N}_t = x^{N}_{N^2t}/N$ to satisfy the equation 
$$v_t = 2\int_0^t W'_\varepsilon(v_s) \frac{\Phi(\rho(^{(\varepsilon)}(s,v_s)))}{\rho^{(\varepsilon)}(s,v_s)}ds + M(t),$$
 where $M$ is a time-changed Brownian motion.  When the zero-range process consists of independent motions, that is when $g(k)\equiv k$ and so $\Phi(u)\equiv u$, one recovers a form of Brox diffusion in terms of $W'_\varepsilon$ instead of white noise $W'$.  We leave to a later work a rigorous study of these considerations.

\section{Model description}
We first introduce the random environments considered, and then the zero-range process in this random environment.
Let $\left\{ r_x  \right\}_{x\in\N}$ be a sequence of i.i.d.\,random variables with mean $0$ and variance $\sigma^2 <\infty$.

Let $s_0 = 0$ and for $n\geq 1$, $s_n = \sum_{k=1}^n r_k $. For
$0\leq u \leq 1$, let
\begin{equation}
\label {eqn: X^N}
X^N_u =
 \dfrac{1}{\sigma \sqrt N} 
 s_{\lfloor Nu \rfloor}
+
\dfrac{Nu - \lfloor Nu \rfloor} {\sigma \sqrt N} r_{\lfloor Nu \rfloor
  + 1}\;, 
\end{equation}
where $\lfloor a \rfloor$, $a\in \R$, stands for the integer part
of $a$.  It is standard that $\left\{X^N_u: 0\leq u\leq 1\right\}$, as
a random function on $[0,1]$, converges in distribution to the
Brownian motion on $[0,1]$.  By Skorokhod's Representation Theorem, we
may find a probability space $(\Omega, \Pa)$ and
$\left\{W^N_u: 0\leq u\leq 1\right\}$, $N\in\N$, mappings from
$\Omega$ to $C[0,1]$, such that, for all $N\in\N$,
$$\left\{X^N_u: 0\leq u\leq 1\right\}=\left\{W^N_u: 0\leq u\leq
1\right\}$$ 
in distribution and moreover,
$\left\{W^N_u: 0\leq u\leq 1\right\}$ converges almost surely to the
standard Brownian motion $\left\{W_u, 0\leq u\leq 1\right\}$.

\medskip
\noindent {\it Quenched formulation.} We will now fix throughout the paper an $\omega\in \Omega$ such that 

\noindent $\left\{W^N_u(\omega): 0\leq u\leq 1\right\}$ converges (uniformly) to a Brownian path $\left\{W_u(\omega): 0\leq u\leq 1\right\}$. 

\medskip

Define the discrete torus $\T_N:=\Z\setminus N\Z$ for $N\in\N$.
Throughout this article, we will identify $\T_N$ with $\{1,2,\ldots,N\}$
and also identify the unit torus $\T$ with $(0,1]$. Fix any $\e$ such that $0<\e<1$.

 It will be convenient to extend $W_t^N$ as well as $W_t$ to $t\in[-1,2]$:  With $\tilde W_u$ representing either $W^N_u$ or $W_u$, we have
\[
\tilde W_u =
\begin{cases}
\tilde W_{u+1} - \tilde W_1& u\in [-1,0),\\
\tilde W_{u-1} + \tilde W_1& u\in (1,2].
\end{cases}
\]

Let $\varepsilon>0$ be a parameter, fixed throughout the paper. For
each $N\in \N$ and $k\in \T_N$, define an $\e$-average of local
environments:
\begin{equation}
\label{eqn: q and W}
q_k^N := \frac{\sqrt{N}}{2\varepsilon N +1} \big(W^N_{\frac{k+\lfloor \varepsilon N\rfloor}{N}} - W^N_{\frac{k-\lfloor \varepsilon N\rfloor}{N}}\big) \stackrel{d}{=} \dfrac1{(2 \e N +1)\sigma} \sum_{|j-k|\leq \e N} r_j. 
\end{equation}
In particular, when $k/N\rightarrow x\in \T$ as $N\uparrow\infty$, we have
\begin{equation}
\label{q_conv}
\sqrt{N} \, q_k^N \rightarrow \frac{1}{2\varepsilon}
\big(W_{x+\varepsilon} - W_{x-\varepsilon}\big).
\end{equation}

The following estimate, afforded by the uniform convergence of
$W^N_\cdot$ to $W_\cdot$, will be useful.

\begin{Lem} \label {lem: qx bound}
There exists constant $C=C(\omega)<\infty$ such that 
\begin{equation} \label {eqn: q bound}
\max_{1\leq k\leq N} \left\{\sqrt N \, |q_k^N| \right\} \leq C\;.
\end{equation}
\end{Lem}

\begin{proof}
Notice that
$ \sqrt N |q_k^N| = \dfrac{N}{2 \e N +1} \big| W_{(k+\e N)/N}^N -
W_{(k-\e N-1)/N}^N \big| $.  The lemma follows from the uniform
convergence of $W_\cdot^N \to W_\cdot$ and the continuity of $W$.
\end{proof}

We now introduce the zero-range process in the random environment
$\{q_k^N : k\in \T_N\}$.  Set $\N_0 = \{0,1,2,\ldots\}$, and let
$\Sigma_N = \N_0^{\T_N}$ be the configuration space. Elements of
$\Sigma_N$ are represented by the Greek letter $\xi$. Thus, $\xi(k)$,
$k\in \T_N$, stands for the number of particles at site $k$ for the
configuration $\xi$.

Fix a function $g:\N_0 \to \R_+$.  Denote by $(\xi_t : t\ge 0)$ the
continuous-time Markov chain whose evolution can be informally
described as follows. Take $N$ sufficiently large for
$|\, q_k^N\,|/\sqrt N <1/2$ for all $k\in\T_N$. At rate
$g(\xi(k)) \,\{\, (1/2) \pm (q_k^N/\sqrt N)\,\}$ a particle jumps from
$k$ to $k \pm 1$.

More precisely, the process $\{\xi_t: t\geq 0\}$ is the Markov process
with generator $L$ given by
\begin{equation}
\label{eqn: generator L}
\begin{split}
Lf(\xi) =
\sum_{k=1}^N 
\Big\{
g(\xi(k)) \Big( \dfrac12 + \dfrac{q_k^N}{\sqrt N}\Big)
&\big(f(\xi^{k,k+1}) - f(\xi) \big)\\
+&
g(\xi(k)) \Big( \dfrac12 - \dfrac{q_k^N}{\sqrt N}\Big)
\big(f(\xi^{k,k-1}) - f(\xi) \big)
\Big\}.
\end{split}
\end{equation}
Here, $\xi^{j,k}$ is the configuration obtained from $\xi$ by moving a
particle from $j$ to $k$, that is,
\begin{equation*}
\xi^{j,k} (\ell) \;=\;
\begin{cases}
\xi(j) -1 & \ell=j\;, \\
\xi(k)+1 & \ell=k \;, \\
\xi(\ell) & \ell\not = j\,, k \;.
\end{cases}
\end{equation*}

To avoid degeneracies, we suppose that $g(0)=0$ and $g(k)>0$ for
$k\geq 1$.  We assume, further, that $g$ satisfies
\begin{enumerate}
\item $\sup_{k\in \N}|g(k+1) - g(k)| \leq g^* <\infty$;
\item There exists $k_0\in \N$ and $c_1>0$ such that $g(k)-g(j)\geq c_1$ for all $k\geq j+k_0$.
\end{enumerate}
These properties guarantee that the zero-range process has good mixing
properties, useful in the proof of the `Replacement Lemma', stated in
Lemma \ref{replacement_lemma}.  This specification also implies that
$g_* k\leq g(k)\leq g^* k$ for some $g_*>0$, more than enough to
satisfy the technical condition `FEM' (cf. \cite{KL}[p.69]) used for
particle truncation in the $1$ and $2$-blocks estimates presented in
Lemmas \ref{lem: 1 block} and \ref{lem: 2 blocks}.

\subsection{Invariant measure} \label{subsec: invariant measure} The
building blocks for the invariant measures are $\{\Pb_{\phi}\}$, a
family of Poisson-like distributions indexed by $\phi\geq 0$
(sometimes referred as fugacity).  For each $\phi$, $\Pb_{\phi}$ is
defined by
\[
\Pb_{\phi}(n) =
\dfrac{1}{Z(\phi)} \dfrac{\phi^n} {g(n) !},
\quad \text{for } n\geq 0.
\]
Here, $g(0)! :=1$ and $g(n)! :=\prod_{j=1}^n g(j)$ for $n\geq 1$;
$Z(\cdot)$ is the partition function:
\[
Z(\phi) := \sum_{n=0}^\infty  \dfrac{\phi^n} {g(n) !}.
\]
Since $g(k) \ge g_* \, k$, $Z(\phi) <\infty$ for all $0\le \phi <\infty$.

Let $R(\phi) = E_{\Pb_\phi} [ X ]$ be the mean of the distribution
$\Pb_\phi$. A direct computation yields that $R'(\phi)>0$, $R(0)=0$,
$\lim_{\phi\to\infty} R(\phi) = \infty$. Since $R$ is strictly
increasing, it has an inverse, denoted by $\Phi: \R_+ \to \R_+$, and
we may parametrize the family of distributions $\Pb_{\phi}$ by its
mean. For $\rho\ge 0$, let $\Qb_\rho = \Pb_{\Phi(\rho)}$, so that
$E_{\Qb_\rho} [ X ] = E_{\Pb_{\Phi(\rho)}} [ X ] = R(\Phi(\rho)) = \rho$.

A straightforward computation yields that
$E_{\Pb_\phi} [ g(X) ] = \phi$, $\phi\ge 0$. Thus,
\begin{equation}
\label{Phi_eqn}
\Phi(\rho) = E_{\Pb_{\Phi(\rho)}} [ g(X) ]
= E_{\Qb_{\rho}} [ g(X) ]\ \;, \quad \rho\,\ge\, 0\;.
\end{equation}
As $g_* k \leq g(k) \leq g^*k$, we have that
$g_*\rho\leq \Phi(\rho) \leq g^*\rho$. On the other hand, a simple
computation yields that $\Phi'(\rho) = \Phi(\rho)/\sigma^2(\rho)$
where $\sigma^2(\rho)$ is the variance of $X$ under $\Qb_\rho$.
Moreover, under our assumptions on $g$, there exist constants
$0<C_1<C_2<\infty$ such that $0<C_1\leq \Phi'(\rho)\leq C_2<\infty$
for all $\rho\ge 0$ (cf. \cite{LSV}[equation (5.2)]). In particular,
$\Phi$ is a strictly increasing function of $\rho$.

Fix a vector $( \phi_{k,N} : k\in\T_N)$ of non-negative real numbers.
Denote by $\Pam_N$ the product measure on $\N_0^{\T_N}$ whose
marginals are given by
\[
\Pam_N (\xi(k) = n) = \Pb_{\phi_{k,N}}(n),
\quad \text{for } k\in \T_N, n\geq 0.
\]
It is standard (cf.~\cite{A}) to check that $\Pam_N$ is invariant with
respect to the generator $L$ in \eqref{eqn: generator L} as long as
the fugacities $\{\phi_{k,N}\}_{k\in \T_N}$ satisfy:
\begin{equation} \label {eqn: invariant phi}
\Big( \dfrac12 + \dfrac{q_{k-1}^N}{\sqrt N}\Big)\phi_{k-1,N}
+
\Big( \dfrac12 - \dfrac{q_{k+1}^N}{\sqrt N}\Big)\phi_{k+1,N}
=
\phi_{k,N},
\quad
k=1,2,\ldots,N.
\end{equation}

Notice that $\{c \phi_{k,N}\}_{k\in \T_N}$, $c\in \R$, is a solution
of \eqref{eqn: invariant phi} if $\{\phi_{k,N}\}_{k\in \T_N}$ is a solution.
In particular, any solution gives rise to a one-parameter family of
solutions.

\begin{Lem}
\label{l01}
The equation \eqref {eqn: invariant phi} admits a solution,
unique up to a multiplicative constant. Moreover, the solution is
either strictly positive or strictly negative or identically equal to
$0$. 
\end{Lem}

\begin{proof} 
Let $\mf r_k = (1/2) + (q_k^N/\sqrt N)$,
$\mf l_k = (1/2) - (q_k^N/\sqrt N) $, $k\in \bb T_N$.  With this
notation, equation \eqref{eqn: invariant phi} becomes
\[
\mf r_{k-1}\, \phi_{k-1,N}
+
\mf l_{k+1}\, \phi_{k+1,N}
=
\phi_{k,N} .
\]
Since $\mf r_k + \mf l_k = 1$, we have that
\[
\mf r_{k-1}\phi_{k-1,N} - \mf l_k \phi_{k,N} 
=
\mf r_k \phi_{k,N}- \mf l_{k+1}\phi_{k+1,N},
\quad
k\in \T_N\; .
\]
Denote by $\gamma$ be the common value of
$\mf r_k \phi_{k,N}- \mf l_{k+1}\phi_{k+1,N}$, to get the recursive
equation
\begin{equation}
\label{03}
\phi_{k+1,N} = \dfrac{\mf r_k \phi_{k,N} - \gamma}{\mf l_{k+1}}\;\cdot 
\end{equation}

With the convention that $\mf r_{N+1} = \mf r_1$ and $\mf l_{N+1} = \mf l_1$, let
\begin{equation*}
\mf R_{j,k} \;=\;  \prod_{i=j}^{k} \mf r_i \;, \qquad 
\mf L_{j,k} \;=\;  \prod_{i=j}^{k}  \mf l_i \;,
\quad 1 \,\le\,  j \,\le\,  k \,\le\,  N+1\;.
\end{equation*}
We extend the definition to indices $j>k$ by setting
$\mf R_{j,k} = \mf L_{j,k} = 1$ if $j>k$.  Solving the recursive
equation yields that
\[
\phi_{k,N}= \dfrac{\mf R_{1,k-1}}{\mf L_{2,k}}\, 
\phi_{1,N}
\;-\; \frac{ \mf S_{k} } {\mf L_{2,k}}\, \gamma \;,
\quad 2\,\le\, k\, \le\, N+1\; . 
\]
In this formula, 
\begin{equation*}
\mf S_{k} \;=\; \sum_{j=2}^k  \mf L_{2,j-1} \, \mf R_{j,k-1}\;,
\end{equation*}
with the convention, adopted above, that
$\mf L_{2,1} = \mf R_{k,k-1}=1$.

Since $\phi_{N+1,N}= \phi_{1,N}$, we have that
\begin{equation}
\label{04}
\gamma \;=\; \frac{\mf R_{1,N} - \mf L_{2,N+1}}{\mf S_{N+1}}\,\phi_{1,N} \;.
\end{equation}
Reporting this value in the equation for $\phi_{k,N}$ yields that
\begin{equation}
\label{01}
\phi_{k,N} \;=\; \Big\{ \dfrac{\mf R_{1,k-1}}{\mf L_{2,k}}\;-\;
\frac{ \mf S_{k} }{\mf S_{N+1}} \, 
\frac{\mf R_{1,N} - \mf L_{2,N+1}} {\mf L_{2,k}}\, \Big\}\,
\phi_{1,N} \;, \quad 2\,\le\, k\, \le\, N \;. 
\end{equation}

Therefore, for each $\phi_{1,N} \in \bb R$, the solution of the
difference equation \eqref{eqn: invariant phi} is given by
\eqref{01}. This proves existence and uniqueness up to a
multiplicative constant. Moreover, it is not difficult to check that
$\mf S_{k} \, \mf R_{1,N} \le \mf S_{N+1} \mf R_{1,k-1}$. Therefore,
as each variable $\mf r_i$, $\mf l_j$ is strictly positive for
for sufficiently large $N$,
the solution is strictly positive if
$\phi_{1,N}>0$. 
\end{proof}

Let
$ \phi_{\max,N} = \max_{1\leq k\leq N} \left\{\phi_{k,N} \right\}$
and
$ \phi_{\min,N} = \min_{1\leq k\leq N} \left\{\phi_{k,N}
\right\}$.

\begin{Lem} 
\label {lem: uniform bounds on phi max min}
Let $\phi_{k,N}$ be a solution of \eqref{eqn: invariant phi}. Then,
there exist constants $C_1=C_1(\omega), C_2=C_2(\omega)<\infty$ such
that for all $N\in \N$
\[
 \dfrac{\phi_{\max,N}}{\phi_{\min,N}} \leq C_1\quad
 \text{and} \quad
  \max_{1\leq k\leq N} |\phi_{k,N} - \phi_{k+1,N}|
 \leq \frac{C_2}{N}\, \phi_{\max,N}.
\]
\end{Lem}

\begin{proof}
By Lemma \ref{lem: qx bound}, there exists a finite constant $C$ such that 
\[
\max_{1\leq k\leq N} \left\{ |q_k^N| \right\} \leq C / \sqrt N.
\]
Therefore, for each $j$, $k\in \bb T_N$,
$|\, (\mf r_j/\mf l_k) \,-\, 1\,|$ and $|\, (\mf l_k/\mf r_j) \,-\, 1\,|$
are bounded from above by
\begin{equation}
\label{02}
2\, \frac{\max_\ell\, |q_\ell^N|}{\sqrt N}
\, \dfrac{1}{\dfrac12 - \dfrac{\max_\ell |q_\ell^N|}{\sqrt N}}
\; \leq\; \dfrac{8C}{N}, \text{ for $N\geq 4C$}. 
\end{equation}

Without loss of generality, we may assume that
$\phi_{1,N} = \phi_{\min,N}$, $\phi_{m,N} = \phi_{\max,N}$. Hence, by
\eqref{01} and since all terms are positive and
$\mf L_{2,N+1} = \mf L_{2,m} \, \mf L_{m+1,N+1}$,
$\phi_{\max,N}/\phi_{\min,N}$ is bounded by
\begin{equation*}
\dfrac{\mf R_{1,m-1}}{\mf L_{2,m}}\;+\;
\frac{ \mf S_{m} \, \mf L_{m+1,N+1}}{\mf S_{N+1}} \;.
\end{equation*}
By \eqref{02},
the first term is bounded by $[1+(8C/N)]^N$. 
We further show that the second term is also bounded by
$[1+(8C/N)]^N$ by rewriting
$\mf L_{m+1,N+1}$ as $\mf R_{m,N} \, [\mf L_{m+1,N+1} / \mf R_{m,N}]$ and using
$\mf S_{m} \, \mf R_{m,N} \le \mf S_{N+1}$. This proves the first
assertion of the lemma.

We turn to the second assertion. By \eqref{03},
\begin{equation*}
\phi_{m+1,N} \;-\; \phi_{m,N}  \;=\;
\Big( \, \frac {\mf r_m}{\mf l_{m+1}} \,-\, 1\,\Big)\, \phi_{m,N}
\;-\; \frac{\gamma}{\mf l_{m+1}}\;\cdot 
\end{equation*}

By \eqref{02}, the absolute value of the first term on the right-hand
side is bounded by $(8C/N)\, \max_{k} \phi_{k,N}$. By \eqref{04}, and
since $\mf L_{2,N+1} = \mf L_{1,N}$, the second one is equal to
\begin{equation*}
\frac{\phi_{1,N}}{\mf l_{m+1}}\, \frac{1}{\mf S_{N+1}}\, 
\sum_{j=0}^{N-1} \{\, \mf L_{1,j}\,
\mf R_{j+1,N} \,-\, \mf L_{1,j+1}\,\mf R_{j+2,N}\, \}\;,
\end{equation*}
where we used the convention that $\mf L_{i,j}=1$ and $\mf R_{i,j}=1$
if $i>j$. Changing variables this expression becomes
\begin{equation*}
\frac{\mf l_1}{\mf l_{m+1}}\, \frac{\phi_{1,N}}{\mf S_{N+1}}\,
\sum_{j=2}^{N+1} \mf L_{2,j-1}\,
\mf R_{j,N} \, \Big( \,  \frac{\mf r_{j-1}}{\mf l_{j-1}} \,-\,
1\Big)  \;,
\end{equation*}
There is $C_0>0$ such that $\mf l_1 / \mf l_{m+1} \leq C_0$. Also by \eqref{02}, the absolute value of the expression inside the
parenthesis is bounded by $8C/N$, uniformly over $j$. The remaining sum
is bounded by $\mf S_{N+1}$. This expression is, therefore, bounded by
$C_0(8C/N)\, \phi_{1,N}$. To complete the proof of the second assertion
of the lemma, it remains to recollect all previous estimates.
\end{proof}

\section{Results}
\label{sec03}

We first specify the initial measures for the zero-range
processes. These include the usual `local equilibrium' measures as
well as others.  We then state the main result of this work.

\subsection{Initial measures}

We consider an initial macroscopic density profile
$\rho_0(\cdot)\in L^1(\T)$, and an initial microscopic measure
satisfying the following condition. Denote by $\Pam_N$ the invariant
measure chosen so that $\max_{k\in\T_N} \left\{\phi_{k,N} \right\}=1$.

\begin{Cond} \label{cond: initial measure} Let $\{\mu^N\}_{N\in \N}$
be a sequence of probability measures on $\bb N_0^{\T_N}$
such that
\begin{enumerate}
\item[(a)] $\{\mu^N\}_{N\in \N}$ is associated with profile $\rho_0$ in the
sense that for any $G\in C(\T)$ and $\delta>0$
\begin{equation*}
\lim_{N\to \infty} \mu^N
\Big[ \, \Big|
\dfrac1N \sum_{k=1}^N G\Big(\dfrac kN \Big) \xi(k)
-
\int_\T G(x) \rho_0(x) dx
\Big|
>\delta
\, \Big]
 =0\; .
\end{equation*}
\item[(b)] The relative entropy of $\mu^N$ with respect to $\Pam_N$ is of
order $N$. There exists a finite constant $C_0$ such that
$ H(\mu^N | \Pam_N):= \int f_0 \ln f_0 \,d\Pam_N \le C_0 N$
\end{enumerate}
for all $N\ge 1$, where $f_0 = d\mu^N/d\Pam_N$.
\end{Cond}

A useful consequence of the relative entropy bound in part (b) of
Condition \ref{cond: initial measure} and the bounds on the fugacities
of the invariant measure $\Pam_N$ in Lemma \ref{lem: uniform bounds on
  phi max min} is that the expected number of particles under $\mu^N$
is of order $N$. Indeed, by the entropy bound,
\begin{equation}
\label{expected_number}
E_{\mu^N}\big[\sum_{k\in \T_N} \xi(k)\big] \;\le\;
\frac{1}{\gamma}H(\mu^N| \Pam_N) \;+\;
\frac{1}{\gamma}\log E_{\Pam_N}\big[e^{\gamma\sum_{x\in \T_N}\xi(k)}] 
\end{equation}
for all $\gamma>0$. By Condition 3.1.b and by definition of $\Pam_N$,
this expression is bounded by
\begin{equation*}
\frac{C_0}{\gamma}\,  N  \;+\; \frac{1}{\gamma} \sum_{k=1}^N
\log \frac{Z(e^\gamma \phi_{k,N})}{Z(\phi_{k,N})}\;\cdot
\end{equation*}
Since $\max_k \phi_{k,N} = 1$, $\min_k \phi_{k,N} \ge c_0>0$ and $Z$
is an increasing function defined on $\R_+$, choosing, say,
$\gamma=1$, yields that the previous expression is bounded by $C_0N$
for some finite constant $C_0$.

Condition \ref{cond: initial measure} is satisfied, for example, by
`local equilibrium' measures
$\left\{\mu^N_{\text{le}}\right\}_{N\in\N}$ associated to macroscopic
profiles $\rho_0$ in $L^\infty(\T)$.  For each $N\in \N$, let
$\mu^N_{\text{le}}$ be the product measure on
$\N_0^{\T_N}$ with marginals given by
\[
\mu^N_{\text{le}} (\xi(k) = n)  \;=\;
\Pb_{\tilde\phi_{k,N}}(n), \quad \text{for } k\in \T_N \,,\,  n\geq 0
\;, 
\]
where the parameters $\{ \tilde\phi_{k,N} : 1\le k\le N\}$ are such
that $E_{\mu^N_{\text{le}}} [\xi(k)] = \rho_{k,N}$ for
\[ \rho_{k,N} = N\int_{(k-1)/N}^{k/N} \rho_0(x) dx.\]

To prove Condition \ref{cond: initial measure}.a,
approximate the integral by a Riemann sum and apply Chebyshev and
Schwarz inequality, keeping in mind that the measure
$\mu^N_{\text{le}}$ is product. Next lemma asserts that condition (b)
is also in force. 

\begin{Lem}
There exists $C_0=C_0(\omega)>0$ such that $H(\mu^N_{\text{le}} | \Pam_N) \leq C_0 N$ for all $N\in \N$.
\end{Lem}
\begin{proof}  Write
\[
\begin{split}
H(\mu^N_{\text{le}} | \Pam_N)
=&
\sum_{k=1}^N E_{\mu^N_{\text{le}}} \left\{
\ln \left(\dfrac{\tilde \phi_{k,N}}{\phi_{k,N}} \right)^{\eta(k)}
+ \ln \dfrac{Z(\phi_{k,N} )}{Z(\tilde \phi_{k,N})} \right\}  \\
=&
\sum_{k=1}^N 
\rho_{k,N} \ln \dfrac{\tilde \phi_{k,N}}{\phi_{k,N}}
+
\sum_{k=1}^N   \ln \dfrac{Z(\phi_{k,N})}{Z(\tilde \phi_{k,N})}.
\end{split}
\]
Then, the lemma follows as (a) $\tilde \phi_{k,N}\leq g^*\rho_{k,N}$ is uniformly bounded by the fugacity bounds after \eqref{Phi_eqn} as $\|\rho_0\|_\infty<\infty$, (b) $Z(0)=1$, and (c) $0<c\leq \phi_{k,N}\leq 1$ for all $1\leq k \leq N$ by Lemma \ref{lem: uniform bounds on phi max min}.
\end{proof}

\subsection{Main result}
For each $N$, we will observe the evolution speeded up by $N^2$, and consider in the sequel the process $\eta_t: = \xi_{N^2t}$, generated by $N^2L$, for times $0\leq t\leq T$, where $T>0$ refers to a fixed time horizon.  
We will access the space-time structure of the process through the scaled mass empirical measure:
\begin{equation*}
\pi_t^N(dx)
:=
\dfrac 1 N 
\sum_{k=1}^N
\eta_t (k) \delta_{k/N}(dx)\;,
\end{equation*}
where $\delta_x$, $x\in\T$, stands for the Dirac mass at $x$.

Let $\Mb$ be the space of finite nonnegative measures on $\T$, and
observe that $\pi^N_t\in \Mb$.  We will place a metric
$d(\cdot,\cdot)$ on $\Mb$ which realizes the dual topology of $C(\T)$
(see \cite{KL}[p. 49] for a definitive choice).  Here, the
trajectories $\{\pi^N_t: 0\leq t\leq T\}$ are elements of the
Skorokhod space $D([0,T],\Mb )$, endowed with the associated Skorohod
topology.

In the following, for $G\in C(\T)$ and $\pi\in \Mb$, denote
$\langle G, \pi\rangle = \int_0^1 G(u)\pi(du)$.  Also, for a given
measure $\mu$, we denote expectation and variance with respect to
$\mu$ by $E_\mu$ and $\text{\rm Var}_\mu$.  Also, the process measure
and associated expectation governing $\eta_\cdot$ starting from $\mu$
will be denoted by $\P_\mu$ and $\E_\mu$.  When the process starts
from $\{\mu^N\}_{N\in \N}$, in the class satisfying Condition
\ref{cond: initial measure}, we will denote by $\P_N:= \P_{\mu^N}$ and
$\E_N:= \E_{\mu^N}$, the associated process measure and expectation.

We are now ready to state our main result.
\begin{Thm} \label{main thm}
For initial measures $\mu^N$ satisfying Condition \ref{cond: initial measure}, consider the speeded process $\eta_t$ as above. 
Then, for any $t\geq 0$, test function $G\in C^{\infty}(\T)$, and $\delta>0$,
\begin{equation*}
\lim_{N\to \infty} \P_N
\Big[ \Big|
\langle G, \pi^N_t \rangle
-
\int_\T G(x) \rho(t,x) dx
\Big|
>\delta
 \Big]
 =0,
\end{equation*}
where $\rho(t,x)$ is the unique weak solution of 
\begin{equation} \label{eqn: with ave}
\begin{cases}
\partial_t \rho(t,x)
=
\dfrac12 \partial_{xx} \Phi(\rho(t,x))
-
2 \partial_x \left(W'_\varepsilon(x)
\Phi(\rho(t,x))\right),\\
\rho(0,x) = \rho_0(x),
\end{cases}
\end{equation}
in the sense of Definition \ref{weak_def}, where $W'_\varepsilon(x) = \dfrac{W_{x+\e} - W_{x-\e}}{2\e}$ in terms of a fixed $\varepsilon>0$.
\end{Thm}

\section{Stochastic differentials and martingales} \label{section: martingale}
To analyze $\langle G, \pi^N_t\rangle$, we compute its stochastic differential in terms of certain martingales.  Let $G$ be a smooth function on $[0,T] \times \T$, and let us write $G_t(x) := G(t,x)$.  Then,
\begin{equation*}
M^{N,G}_t
=
\left\langle G_t,   \pi^N_t \right\rangle
-
\left\langle G_0,  \pi^N_0 \right\rangle
-
\int_0^t 
\Big\{
 \left\langle  \partial_s G_s, \pi_s^N\right\rangle + N^2L \left\langle G_s,   \pi^N_s \right\rangle 
 \Big\}ds
\end{equation*}
is a mean zero martingale.
Denote the discrete Laplacian $\Delta_{N}$ and discrete gradient $\nabla_N$ by
\begin{equation*}
\begin{split}
\Delta_N G\Big(\frac k N\Big)
:=&
N^2 \Big(G\Big(\frac {k+1}N\Big)+G\Big(\frac {k-1}N\Big)-2G\Big(\frac k N\Big)\Big),\\
\nabla_N G\Big(\frac k N\Big)
:=&
\dfrac N2  \Big(G\Big(\frac {k+1}N\Big) - G\Big(\frac {k-1}N\Big)\Big),
\end{split}
\end{equation*}
and write
\begin{equation} \label{gen_comp}
\begin{split}
&N^2L \left\langle G,  \pi^N_s \right\rangle 
= 
\dfrac1N \sum_{1\leq k\leq N}
\left(
\dfrac12 \Delta_{N} G_s\Big(\frac k N \Big) g(\eta_s(k))
+
2\nabla_{N} G_s\Big( \frac k N \Big) g(\eta_s(k)) \sqrt{N} q_k
\right).
\end{split}
\end{equation}
We will define
\begin{equation}
\label{DG}
D^{G,s}_{N,k} 
:=
\dfrac12 \Delta_{N} G_s\Big(\frac k N \Big)
+
2\nabla_{N} G_s\Big( \frac k N \Big) \sqrt{N} q_k^N.
\end{equation}
As  $\sqrt{N} \vert q_k^N \vert$ is uniformly bounded from above by a constant $C$, cf.\,\eqref{eqn: q bound}, we have
\begin{equation}
\label{DG_bound}
\left |D^{G,s}_{N,k}\right | \leq
\|\partial_{xx} G\|_\infty
 +
2C\|\partial_x G\|_\infty.
\end{equation}

The quadratic variation of $M^{N,G}_t$ is given by
\begin{equation*}
\langle M^{N,G} \rangle_t
=
\int_0^t 
\left\{
N^2L \left(\left\langle G_s,  \pi^N_s \right\rangle ^2 \right)
-
2 \left\langle G_s,  \pi^N_s \right\rangle
 N^2 L \left\langle G_s,  \pi^N_s\right\rangle 
\right\}
ds.
\end{equation*}
Standard calculation shows that
\begin{equation*} \label {quadratic variation ln k}
\begin{split}
\langle M^{N,G} \rangle_t
=&
\int_0^t 
 \sum_{k=1}^{N}
\left\{
g(\eta_s(k)) \Big( \dfrac12 +\dfrac{q_k^N}{\sqrt N}\Big)
\Big(G_s\Big(\frac {k+1}N\Big) - G_s\Big(\frac kN\Big)\Big)^2
\right.\\
&\quad\quad\quad\quad\quad\quad
+\left.
g(\eta_s(k)) \Big( \dfrac12 -\dfrac{q_k^N}{\sqrt N}\Big)
\Big(G_s\Big(\frac {k-1}N\Big) - G_s\Big(\frac kN\Big)\Big)^2
\right\}
ds.
\end{split}
\end{equation*}
This variation may be bounded as follows.
\begin{Lem} \label{lem: martingale bounds} 
For smooth functions $G$ on $[0,T]\times \T$, there is a constant $C_G$ such that for large $N$,
$$\sup_{0\leq t\leq T}\E_N \langle M^{N,G}\rangle_t \leq g^*  C_G T N^{-1}.$$
\end{Lem}
\begin{proof}
For $N$ large, we may assume that $1/2 +\vert q_k^N /\sqrt N\vert \leq 1$. Also since $G$ is smooth and $g(\cdot)$ grows at most linearly, we obtain
\begin{eqnarray*}
\E_N\langle M^{N,G} \rangle_t
 &\leq &
 2g^* \| \partial_x G\|_\infty N^{-1}
\E_N
 \Big [
 \dsp\int_0^t 
\dfrac1N \sum_{k=1}^N
\eta_s (k)
 ds  \Big ]\\
 &=&
 2g^* \| \partial_x G\|_\infty N^{-1} t\,
\E_N
 \Big [
\dfrac1N \sum_{k=1}^N \eta_0(k)
  \Big ].
\end{eqnarray*}
We have used that total number of particles is conserved in the last equality.  By \eqref{expected_number}, we have that
$\sup_N  E_N \big [ N^{-1} \sum_{k=1}^N \eta_0(k)\big ]< \infty$, and the result follows.  \end{proof}

\section{Proof outline}
We now outline the proof of Theorem \ref{main thm}.
Let  $Q^N$ the probability measure on the trajectory space $D([0,T],\Mb )$ 
governing $\pi_{\cdot}^N$ when the process starts from $\mu^N$.
By Lemma \ref{lem: tightness}, the family of measures $\left\{Q^N\right\}_{N\in \N}$ is tight with respect to the uniform topology, stronger than the Skorokhod topology. 

Let now $Q$ be any limit measure. 
We will show that $Q$ is supported on a class of weak solutions to the nonlinear PDE \eqref{eqn: with ave}.
\vskip .1cm

{\it Step 1.}
Let $G$ be smooth on $[0,T]\times \T$.
Recall the martingale $M^{N,G}_t$ and its quadratic variation $\langle M^{N,G}\rangle_t$ in the last section.
By Lemma \ref{lem: martingale bounds}, we have $\E_N \big( M^{N,G}_T\big)^2 = \E_N \langle M^{N,G} \rangle_T$
vanishes as $N\to \infty$.
By Doob's inequality, for each $\delta>0$,
\begin{eqnarray*}
&&
\P_N
\Big[
\sup_{0\leq t\leq T} 
\big| 
\big\langle G_t,  \pi^N_t \big\rangle
-
\big\langle G_0,  \pi^N_0 \big\rangle  
-
 \int_0^t \big( \big\langle \partial_s G_s,  \pi^N_s \big\rangle + N^2L  \big\langle G_s,  \pi^N_s \big\rangle   \big)   ds \big|
>\delta
\Big]\\
&& \leq  \dfrac{4}{\delta^2}
 \E_N \big\langle M^{N,G}\big\rangle_T \ \rightarrow \ 0 \ \ {\rm as \ } N\to\infty.
\end{eqnarray*}
Recall the evaluation of $N^2L  \left\langle G_s,  \pi^N_s \right\rangle$ in \eqref{gen_comp}.
Then,
\begin{eqnarray} \label {eqn before replace}
&&
\lim_{N\to \infty}
\P_N
\Big[
\sup_{0\leq t\leq T} 
\Big| 
\left\langle G_t,  \pi^N_t \right\rangle
-
\left\langle G_0,  \pi^N_0 \right\rangle
-
\int_0^t \Big\{ \left\langle \partial_s G_s,  \pi^N_s \right\rangle 
\\
&&
 \dfrac1N \sum_{1\leq k\leq N}
\left(
\dfrac12 \Delta_{N} G_s\Big(\frac k N \Big) g(\eta_s(k))
+
2\nabla_{N} G_s\Big( \frac k N \Big) g(\eta_s(k)) \sqrt{N} q_k^N
\right) 
   \Big\}   ds \Big|
>\delta
\Big]
=0.
\nonumber
\end{eqnarray}
\vskip .1cm

{\it Step 2.}
We now replace the nonlinear term $g(\eta_s(k))$
by a function of the empirical density of particles.
To be precise,
let $\eta^l (x) = \dfrac{1}{2l+1} \sum_{|y-x|\leq l} \eta(y)$,
that is the average density of particles in the box centered at $x$ with length $2l+1$.  

Recall the coefficient $D^{G,s}_{N,k}$ in \eqref{DG}.
By the triangle inequality, the $1$ and $2$-block estimates (Lemmas \ref{lem: 1 block} and \ref{lem: 2 blocks})
imply the following replacement lemma.

\begin{Lem} [Replacement Lemma]
\label{replacement_lemma}
For each $\delta>0$, we have
\begin{equation*}
\begin{split}
\limsup_{\theta\to 0} \limsup_{N\to \infty} 
 \P_N
\Big[
\Big|
\dfrac1N \sum_{1\leq k \leq N }
\int_0^T 
D_{N,k}^{G,t}
\Big( g\left (\eta_t (k)\right )
-\Phi\left(\eta_{t}^{\eps N}(k)\right)\Big) dt
\Big|
 \geq \delta
\Big]
=0.
\end{split}
\end{equation*}
\end{Lem}

\vskip .1cm
{\it Step 3.}
For each $\eps>0$, take $\iota_\eps = (2\eps)^{-1} \id_{[-\eps,\eps]}$.
The average density $\eta_{t}^{\eps N}(k)$ is written as a function of the empirical measure $\pi^N_t$ 
\begin{equation*}
\eta_{t}^{\eps N}(k)
=
\dfrac{2\eps N}{2\theta N+1}
\langle \iota_{\eps}(\cdot -  k/N),  \pi^N_t) \rangle.
\end{equation*}

Then, noting the form of $D^{G,s}_{N,k}$ and the quenched convergence \eqref{q_conv}, we may replace $\nabla_N$, $\Delta_N$, and $\sqrt N q_k^N$
by $\partial_x$, $\partial_{xx}$, and $W^\prime_\e \Big(\dfrac kN\Big)$
respectively, and also the sum by an integral.   Hence, we get from \eqref{eqn before replace}
in terms of the induced distribution $Q^N$ that
\begin{equation}
\label{step3 beta 0 eqn}
\begin{split}
&
\limsup_{\eps \to 0}
\limsup_{N\to \infty}
Q^N
\Big[
\Big | 
\left\langle G_T,  \pi^N_T \right\rangle
-
\left\langle G_0,  \pi^N_0 \right\rangle
-
\int_0^T 
\Big\{ \left\langle \partial_s G_s,  \pi^N_s \right\rangle
 \\
&
+
\int_\T
\left( \dfrac12
\partial_{xx} G_s\left(x \right) 
+
2\partial_x G_s\left( x \right) W^\prime_\e (x)
\right)
\Phi\big(\langle \iota_{\eps}(\cdot -  x),  \pi^N_s\rangle \big)
   dx \Big\}  ds \Big |
>\delta
\Big]
=0.
\end{split}
\end{equation}

Taking $N\to \infty$, along a subsequence, as the set of trajectories in \eqref{step3 beta 0 eqn} is open with respect to the uniform topology, we obtain
\begin{equation*}
\begin{split}
&
\limsup_{\eps \to 0}
Q
\Big[
\Big| 
\left\langle G_T,  \pi_T \right\rangle
-
\left\langle G_0,  \pi_0 \right\rangle
-
\int_0^T \Big\{ \left\langle \partial_s G_s,  \pi_s \right\rangle  \\
&
+
\int_\T
\left(
\dfrac12 \partial_{xx} G_s\left(x \right) 
+
2\partial_x G_s\left( x \right) W^\prime_\e (x)
\right)
\Phi\big(\langle \iota_{\eps}(\cdot -  x),  \pi_s\rangle \big)
  dx \Big\}ds
  \Big|
>\delta
\Big]
=0.
\end{split}
\end{equation*}

\vskip .1cm
{\it Step 4.} We show in
Lemma \ref{lem: rho < phi c} that $Q$ is supported on trajectories $\pi_s(dx) =\rho(s,x) dx $ where $\rho\in L^1([0,T]\times \T)$.
To replace $\langle \iota_{\eps}(\cdot - x),  \pi_s \rangle$
by $\rho(s,x)$, it is enough to show, for all $\delta>0$, that
\begin{equation*}
\begin{split}
&
\limsup_{\eps \to 0}
Q
\Big[
\Big| 
\int_0^T 
\int_\T
D_{G,s}
\left(
\Phi\big(\langle \iota_{\eps}(\cdot -  x),  \pi_s\rangle \big)
-
\Phi(\rho(s,x))
\right) dx 
ds \Big|
>\delta
\Big]
=0.
\end{split}
\end{equation*}
where $D_{G,s} = \dfrac12 \partial_{xx} G_s\left(x \right) +2 \partial_x G_s\left( x \right) W^\prime_\e(x)$.
In fact, considering the Lebesgue points of $\rho$, almost surely with respect to $Q$,
\begin{equation*}
\lim_{\eps\to 0}
\int_0^T 
\int_\T
D_{G,s}
\Phi{\langle \iota_{\eps}(\cdot - x),  \pi_s \rangle}
 dx  ds 
 =
 \int_0^T 
\int_\T
D_{G,s}
 \Phi(\rho(s,x))
 dx  ds .
\end{equation*}
Now, we have
\begin{equation*}
\begin{split}
&
Q
\Big[
\left\langle G_T,  \rho(T,x) \right\rangle
-
\left\langle G_0,  \rho(0,x) \right\rangle
-
\int_0^T \Big\{ \left\langle \partial_s G_s,  \rho(s,x) \right\rangle \\
&
+
\int_\T
\left(
\dfrac12 \partial_{xx} G_s\left(x \right) 
+
2 \partial_x G_s\left( x \right) W^\prime_\e(x)
\right)
\Phi(\rho(s,x))
 dx  \Big\}  ds 
   =0
\Big]
=1.
\end{split}
\end{equation*}

\vskip .1cm
{\it Step 5.}
Hence, each $\rho(t,x)$ solves weakly the 
equation 
\[ \partial_t \rho 
=
 \dfrac12 \partial_{xx} \Phi(\rho) -2 \partial_x \big(W^\prime_\e (x)\Phi(\rho)\big).
\]
  As a consequence of the weak formulation, $\rho$ satisfies conservation of mass (cf. Lemma \ref{lem: rho < phi c}): $\int_\T \rho(t,x)dx = \int_\T \rho_0(x)dx$.  Moreover, the initial condition $\rho(0,x) = \rho_0(x)$ holds by Condition \ref{cond: initial measure}.  From convergence of $Q^N$ to $Q$ with respect to the uniform topology, $\rho$ is weakly continuous in time:  Namely, for each test function $G\in C(\T)$, the map $t\mapsto \int_\T G(x)\rho(t,x)dx$ is continuous.
  In addition, in Proposition \ref{prop: weak derivative of Phi}, we show an energy estimate which defines a weak spatial derivative of $\Phi(\rho(t,x))$.

We show in Subsection \ref{section uniqueness} that there is at most one weak solution $\rho$ to \eqref{eqn: with ave}, subject to these constraints (cf. Definition \ref{weak_def}).
 We conclude then that the sequence of $Q^N$ converges weakly to the Dirac measure on $\rho(\cdot,x)dx$.
Finally, as $Q^N$ converges to $Q$ with respect to the uniform topology, we have for each $0\leq t\leq T$ that $\langle G,\pi_t^N\rangle$ weakly converges to the constant $\int_\T G(x)\rho(t,x)dx$, and therefore convergence in probability as stated in Theorem \ref{main thm}.
\qed

\section{$1$-block estimate}

Following the scheme of \cite{JLS} and \cite{FSX}, the `$1$-block'
estimate is obtained by using a Rayleigh-type estimation of a
variational expression derived from a Feynman-Kac bound.  A spectral
gap bound plays an important role in this step.  Since there are
differences here in the context of the random environment, all details
are given.

Recall the generator $L$,
cf.\,\eqref{eqn: generator L}, and the invariant measure $\Pam_N$,
cf.\,Section \ref{sec03}.  As $\Pam_N$ is not reversible with respect
to $L$, we will work with $S$, the symmetric part of $L$:
\[
Sf(\eta) = \dfrac12 \, 
\sum_{k=1}^N 
\Big\{
g(\eta(k))
\mf p_{k,+}^N
\big(f(\eta^{k,k+1}) - f(\eta) \big)
+
g(\eta(k))
\mf p_{k,-}^N
\big(f(\eta^{k,k-1}) - f(\eta) \big)
\Big\}
\]
where
\begin{equation} \label {eqn: jump prob left right}
\mf p_{k,+}^N:=\big( \dfrac12 + \dfrac{q_k^N}{\sqrt N}\big)
+
\dfrac{\phi_{k+1,N}}{\phi_{k,N}}
\big( \dfrac12 - \dfrac{q_{k+1}^N}{\sqrt N}\big),
\quad
\mf p_{k,-}^N:= \big( \dfrac12 - \dfrac{q_k^N}{\sqrt N}\big)
+
\dfrac{\phi_{k-1,N}}{\phi_{k,N} }
\big( \dfrac12 + \dfrac{q_{k-1}^N}{\sqrt N}\big).
\end{equation}
Then, $\Pam_N$ is reversible under the generator $S$.
The Dirichlet form is
\begin{equation*}
\begin{split}
E_{\Pam_N}
\left[ f(-Sf) \right]
=
E_{\Pam_N}
\left[ f(-Lf) \right]
=
\dfrac12 \, \sum_{k=1}^N
E_{\Pam_N}
\Big[
g(\eta(k))
\mf p_{k,+}^N
\big(f(\eta^{k,k+1}) - f(\eta) \big)^2
\Big].
\end{split}
\end{equation*}

\subsection{Spectral gap bound for $1$-block estimate}
For $k\in\T_N$ and $l\geq 1$, define the set $\L_{k,l}=\left\{k-l,k-l+1,\ldots,k+l\right\}\subset \T_N$. Consider the process restricted to $\L_{k,l}$ generated by
$S_{k,l}$ where
\begin{equation}  \label {eqn: S_k l}
\begin{split} 
S_{k,l}f(\eta)
=
\dfrac12 \, \sum_{x,x+1\in \L_{k,l}}
\Big\{
&g(\eta(x))
\mf p_{x,+}^N
\big(f(\eta^{x,x+1}) - f(\eta) \big) \\
&\quad+
g(\eta(x+1))
\mf p_{x+1,-}^N
\big(f(\eta^{x+1,x}) - f(\eta) \big)
\Big\}.
\end{split}
\end{equation}

Let  $\Omega_{k,l} = \N_0^{\L_{k,l}}$ be the state space of configurations restricted on sites $\L_{k,l}$. For each $\eta\in \Omega_{k,l}$, define
$\kappa_{k,l} (\eta) =  \prod_{x\in \L_{k,l}} \Pb_{\phi_{x,N}}(\eta(x))$,
that is, $\kappa_{k,l}$ is the product measure $\kappa:=\Pam_N $ restricted to $\Omega_{k,l}$.
Define the state space of configurations with exactly $j$ particle on the sites $\L_{k,l}$:
$$\Omega_{k,l,j} = \{\eta \in \Omega_{k,l}: \sum_{x\in \L_{k,l} }\eta(x) = j\}.$$
Let $\kappa_{k,l,j}$ be the associated reversible canonical measure
obtained by conditioning $\kappa_{k,l}$ on $\Omega_{k,l}$.
The corresponding Dirichlet form is
\begin{equation} \label {d form asymm}
E_{\kappa_{k,l,j}}
\left[ f(-S_{k,l} f) \right]
=
\dfrac12 \sum_{x,x+1\in \L_{k,l}}
 E_{\kappa_{k,l,j}}
\Big[
g(\eta(x))
\mf p_{x,+}^N
\big(f(\eta^{x,x+1}) - f(\eta) \big)^2
\Big]. 
\end{equation}

We will obtain the spectral gap estimate corresponding to the localized inhomogeneous process by comparison with the spectral gap for the standard translation-invariant localized process.  
Consider the generator $L_l$ on $\Omega_{k,l}$ given by
\begin{equation*}
\label{onestar}
L_lf(\eta)
=
\sum_{x,x+1\in \L_{k,l}}
\frac{1}{2}\left\{
g(\eta(x)) \left [  f\left(\eta^{x,x+1} \right)  - f(\eta)   \right] 
+g(\eta(x+1))\left [  f\left(\eta^{x+1,x}  \right)   - f(\eta) \right ]
\right\}.
\end{equation*}
For any $\rho>0$, let $\nu_\rho$ be the product measure on $\Omega$ with common marginal $\Pb_{\phi}$ on each site $k\in \N$ with mean $\rho$,
and let $\nu^\rho_{l}$ be its restriction to $\Omega_{k,l}$. 

Consider $\nu_{l,j}$, the associated canonical measure on $\Omega_{k,l,j}$, with respect to $j$ particles in $\L_{k,l}$. Notice that $\nu_{l,j}$ does not depend on $\rho$.  It is well-known that both $\nu^\rho_l$ and $\nu_{l,j}$ are  invariant measures with respect to the localized generator $L_l$ (cf.~\cite{A}).
The corresponding Dirichlet form is given by
\begin{equation} \label {d form symm}
\begin{split}
E_{\nu_{l,j}}
\left[ f(-L_l f) \right]
=&
\dfrac12 \sum_{x,x+1\in \L_{k,l}}
E_{\nu_{l,j}}
 \left[ 
 g(\eta(x)) \left(  f\left(\eta^{x,x+1} \right)  - f(\eta)   \right)^2
 \right].
\end{split}
\end{equation}

We are now ready to state the lemma for the spectral gap bounds.
Recall $\mf p_{k,+}^N$ from \eqref{eqn: jump prob left right}.
Let $ r_{k,l,N}^{-1} := \min_{x\in \L_{k,l}} \left\{ \mf p_{x,+}^N \right\}$.
\begin{Lem} \label {spectral gap ln k}
We have the following estimates:
\begin{enumerate}
\item Uniform bound:  For all $\eta \in \Omega_{k,l,j}$, we have
\begin{equation} \label {eqn: RN for klj ln k case}
\left(\dfrac{\phi_{{\rm min},k,l}}{\phi_{{\rm max},k,l}} \right)^j
\leq
\dfrac{\kappa_{k,l,j} (\eta) }{\nu_{l,j}(\eta)}
\leq
\left(\dfrac{\phi_{{\rm max},k,l}}{\phi_{{\rm min},k,l}} \right)^j
\end{equation}
where 
$\dsp \phi_{{\rm min},k,l} = \min_{x\in\L_{k,l}} \phi_{x,N}$ and $\dsp \phi_{{\rm max},k,l} = \max_{x\in\L_{k,l}} \phi_{x,N}$.
\item Poincar\'e inequality:  We have
\begin{equation*}
{\rm Var}_{\kappa_{k,l,j}} (f)
\leq
 C_{k,l,j} 
 E_{\kappa_{k,l,j}}
 \left[ f(-S_{k,l}f) \right]
\end{equation*}
where 
$
 C_{k,l,j}
:=
C (2l+1)^2 r_{k,l,N}  \left(\dfrac{\phi_{{\rm max},k,l}}{\phi_{{\rm min},k,l}} \right)^{2j}
$ bounds the inverse of the spectral gap of $-S_{k,l}$ on $\Omega_{k,l,j}$
and $C$ is a universal constant.

\item For each $l$ fixed, we have
$$\lim_{N\uparrow\infty} \sup_{1\leq k\leq N} \dfrac{\phi_{{\rm max},k,l}}{\phi_{{\rm min},k,l}} = 1, 
\quad
\lim_{N\uparrow\infty} \sup_{1\leq k\leq N} r_{k,l,N} = 1$$
and hence, for fixed $l$ and $j$, $\dsp \sup_{N\geq 1}\sup_{1\leq k\leq N} C_{k,l,j}<\infty$.
\end{enumerate}
\end{Lem}
\begin{proof}
Fix an arbitrary $\rho>0$. By the definitions of conditioned measures $\kappa_{k,l,j}$ and $\nu_{l,j}$, we have, for $\eta\in \Omega_{k,l,j}$,
\begin{equation} \label{eqn: ratio computation with total j}
\dfrac{\kappa_{k,l,j} (\eta) }{\nu_{l,j}(\eta)}
=
\dfrac{\kappa_{k,l} (\eta) }{\nu^\rho_{l}(\eta)}
\dfrac{\nu^\rho_{l}(\Omega_{k,l,j})}{\kappa_{k,l} (\Omega_{k,l,j})}.
\end{equation}
The product structure of $\kappa_{k,l}$ and $\nu^\rho_{l}$ allows a direct computation 
\begin{equation}
\label{mu nu prod}
\dfrac{\kappa_{k,l} (\eta) }{\nu^\rho_{l}(\eta)}
=
 \dfrac 
 {\prod_{x\in \L_{k,l}} \big\{  \big( \phi_{x,N}\big)^{\eta(x)} /Z( \phi_{x,N})  \big\} }  
 {\prod_{x\in \L_{k,l}} \big\{ (\phi_0)^{\eta(x)}  /Z(\phi_0) \big\} },
 \end{equation}
 where $\phi_0$ is the common fugacity for (the marginals of) $\nu_\rho$.
Recalling that $\phi_{{\rm min},k,l} = \min_{x\in\L_{k,l}} \phi_{x,N}$ and $\phi_{{\rm max},k,l} = \max_{x\in\L_{k,l}} \phi_{x,N}$, for $\eta\in \Omega_{k,l,j}$, we can estimate $\kappa_{k,l} (\eta) / \nu^\rho_{l}(\eta)$ by
\begin{equation} \label{eqn: spectral gap, pointwise ratio}
\left(\dfrac{\phi_{{\rm min},k,l}}{\phi_0} \right)^j
\prod_{x\in \L_{k,l} } \dfrac{Z(\phi_{x,N})}{Z(\phi_0)}
\leq
\dfrac{\kappa_{k,l} (\eta) }{\nu^\rho_{l}(\eta)}
\leq
\left(\dfrac{\phi_{{\rm max},k,l}}{\phi_0} \right)^j
\prod_{x\in \L_{k,l} } \dfrac{Z(\phi_{x,N})}{Z(\phi_0)}.
\end{equation}
Note that $\kappa_{k,l}(\Omega_{k,l,j}) = \sum_{\eta\in \Omega_{k,l,j}}\big[\kappa_{k,l}(\eta)/\nu^\rho_{l}(\eta)\big]\nu^\rho_l(\eta)$.
Then, $\kappa_{k,l} (\Omega_{k,l,j}) / \nu^\rho_{l}(\Omega_{k,l,j}) $ is estimated by the same bounds as in \eqref{eqn: spectral gap, pointwise ratio}.
Then, rearranging these estimates, \eqref{eqn: RN for klj ln k case} follows from 
\eqref{eqn: ratio computation with total j}.

Turning now to the Poincar\'e inequality, 
the proof relies on the well known spectral gap for one dimensional localized symmetric zero range process (cf. \cite{LSV}):
For all $j$, with respect to a universal constant $C$,
\begin{equation}
\label {eqn: spetral gap for nu}
{\rm Var}_{\nu_{l,j}} (f)
\leq
C (2l+1)^2 
E_{\nu_{l,j}}
 [ f(-L_{l}f) ].
\end{equation}
To get an estimate with respect to $-S_{k,l}$,
from \eqref{d form symm} and \eqref{d form asymm}, using \eqref{eqn: RN for klj ln k case}, we have
\begin{equation} \label {eqn: d from comparison}
E_{\nu_{l,j}}
\left[ f(-L_{l}f) \right]
\leq
r_{k,l,N} \left(\dfrac{\phi_{{\rm max},k,l}}{\phi_{{\rm min},k,l}} \right)^j
E_{\kappa_{k,l,j}}
\left[ f(-S_{k,l}f) \right].
\end{equation}
Now, since
\begin{align*}
{\rm Var}_{\kappa_{k,l,j}}(f)
&=
\inf_a E_{\kappa_{k,l,j}}\big[ (f-a)^2 \big]\\
&\leq
\left(\dfrac{\phi_{{\rm max},k,l}}{\phi_{{\rm min},k,l}} \right)^j
\inf_a E_{\nu_{l,j}}\big[ (f-a)^2 \big]
= 
\left(\dfrac{\phi_{{\rm max},k,l}}{\phi_{{\rm min},k,l}} \right)^j
{\rm Var}_{\nu_{l,j}}(f),
\end{align*}
the desired Poincar\'e inequality follows from \eqref{eqn: spetral gap for nu} and \eqref{eqn: d from comparison}.

As the last claim follows from the second assertion of Lemma \ref
{lem: uniform bounds on phi max min}, the proof to the lemma is
complete.
 \end{proof}

\subsection{Relative entropy}
For $t>0$, let $\mu_t^N$ be the distribution of $\eta_t$. As the entropy production is negative, cf.\,p.~340, \cite{KL}, we have $H(\mu_t^N | \Pam_N) \leq H(\mu^N | \Pam_N)\leq C_0 N$.
Furthermore, the relative entropy of $\mu_t^N$ with respect to the homogeneous invariant measures $\nu_\rho$ is also of order $O(N)$ which will be useful in the sequel.
\begin{Lem} \label{lem: relative entropy wrt uniform}
For any fixed $\rho >0$, there is a constant $C=C(\omega)$ such that 
$H(\mu_t^N | \nu_\rho ) \leq C N$.
\end{Lem}
\begin{proof} Write
\[
\begin{split}
H(\mu_t^N | \nu_\rho)
=
\int \ln \left(\dfrac{d\mu_t^N}{d\nu_\rho} \right)d\mu_t^N
=
\int \ln \left(\dfrac{d\mu_t^N}{d\Pam_N} \right)d\mu_t^N
+
\int \ln \left(\dfrac{d\Pam_N}{d\nu_\rho} \right)d\mu_t^N.
\end{split}
\]
The first term on the right-hand side is exactly $H(\mu_t^N | \Pam_N)=O(N)$ by part (2) of Condition \ref{cond: initial measure}. The integrand in the second term equals,
\[
\ln \dfrac{d\Pam_N}{d\nu_\rho} (\eta)
=
\ln \dfrac{\Pam_N(\eta)}{\nu_\rho(\eta)}
=
\sum_{k=1}^N 
\eta(k) \ln \dfrac{\phi_{k,N}}{\phi_0}
+
\sum_{k=1}^N  \ln \dfrac{Z(\phi_0)}{Z(\phi_{k,N})}.
\]
The desired estimate now follows the observations:  $0<c\leq \phi_{k,N}\leq 1$ by 
Lemma \ref{lem: uniform bounds on phi max min}, and the mean expected number of particles, $\int \sum_{k\in \T_N}\eta(k) d\mu^N_t = O(N)$ by \eqref{expected_number}.
\end{proof}

\subsection{$1$-block estimate}
We prove the $1$-block estimate:

\begin{Lem}[$1$-block estimate] 
\label{lem: 1 block}
For every $T>0$,
\[
\limsup_{l\to \infty}
\limsup_{N\to \infty} 
 \E_N
\Big[
\dfrac1N \sum_{1\leq k \leq N }
\Big|
\int_0^T 
 V_{k,l}(s,\eta_s) ds
\Big|
\Big]
=0
\]
where $V_{k,l}(s,\eta)
:=
D_{N,k}^{G,s} 
\left(g(\eta(k))
- 
\Phi(\eta^l(k))\right)$ and $D_{N,k}^{G,s}$ is as in \eqref{DG}.
\end{Lem}

\begin{proof}
We first introduce a cutoff of large densities. Fix $A>0$, and let 
\[
\tilde V_{k,l,A} (s,\eta)
:=
V_{k,l}(s,\eta) \id_{\{\eta^l (k)\leq A\}}.
\]
Notice that $g(\cdot)$ considered here satisfies the `FEM' assumption
in \cite{KL}.  By Lemma \ref{lem: relative entropy wrt uniform}, the
cutoff follows exactly from Lemma 4.2 in p.90, \cite{KL}.

It now remains to prove that for every $A>0$, $T>0$,
\begin{equation*}
\limsup_{l\to \infty} \limsup_{N\to \infty}
\sup_{1\leq k \leq N}
\E_N
\Big[
\Big |
\int_0^T
 \tilde V_{k,l,A} (s,\eta_s) ds
\Big |
\Big]
=
0.
\end{equation*}

Define $\L_{k,l} (\eta)$ be the number of particles in $\L_{k,l}$,
that is $\L_{k,l} (\eta): = (2l+1) \eta^l(k)$.  As in \cite{FSX}, we
will replace $\tilde V_{k,l,A} (s,\eta)$ by its `centering':
\begin{equation*}
V_{k,l,A}(s,\eta)
:=
D_{N,k}^{G,s} 
\Big\{\, g(\eta(k))
- 
E_{\kappa_{k,l, \L_{k,l}(\eta)}}[g(\eta(k))]
\,\Big\}\, \id_{\{\eta^l (k)\leq A\}}.
\end{equation*}
Note that $E_{\kappa_{k,l,j}}  V_{k,l,A}= 0$ for all $k,l,j$ which will be used in the Rayleigh-type estimation.
The error introduced by such a replacement is less than or equal to
\begin{equation}
\label{step2eqn}
\E_N
\Big[
\int_0^T 
\id_{\{0<\eta^l_s (k)\leq A\}}
\Big |
E_{\kappa_{k,l,\L_{k,l}(\eta_s)}}[g(\eta(k))]
-
\Phi(\eta^l_s(k))
\Big |
ds
\Big].
\end{equation}
Note that $\Phi(\eta^l_s(k)) = E_{\nu_{\eta^l_s(k)}}[g]$.
By triangle inequality, \eqref{step2eqn} is bounded by
\begin{eqnarray*}
&&
\E_N
\Big[
\int_0^T
\id_{\{0<\eta^l_s (k)\leq A\}}
\Big |
E_{\kappa_{k,l, \L_{k,l}(\eta_s)}}[ g(\eta(k))]
-
E_{\nu_{k,l,\L_{k,l}(\eta_s)}}[ g(\eta(k))]
\Big |
ds
\Big]\\
&&+
\E_N
\Big[
\int_0^T
\id_{\{0<\eta_s^l (k)\leq A\}}
\Big|
E_{\nu_{k,l,\L_{k,l}(\eta_s)}} [g(\eta(k))]
-
E_{\nu_{\eta^l_s (k)}}[ g ]
\Big |
ds
\Big]
=:I_1+I_2.
\end{eqnarray*}

Using \eqref{eqn: RN for klj ln k case} and then $g(k) \leq g^* k$, the term $I_1$ is bounded by
\[
\begin{split}
&\E_N
\Big[
\int_0^T
\id_{\{0<\eta^l_s (k)\leq A\}}
E_{\nu_{k,l,\L_{k,l}(\eta_s)}}[g(\eta(k))]  
\Big[\left(\dfrac{\phi_{{\rm max},k,l}}{\phi_{{\rm min},k,l}} \right)^{\L_{k,l}(\eta_s)}  - 1 \Big]
ds
\Big]\\
&\leq
T g^* (2l+1)A \Big[\left(\dfrac{\phi_{{\rm max},k,l}}{\phi_{{\rm min},k,l}} \right)^{(2l+1)A}  - 1 \Big].
\end{split}
\]
Notice that $\sup_{1\leq k\leq N}\dfrac{\phi_{{\rm max},k,l}}{\phi_{{\rm min},k,l}}\to 1$ by Lemma \ref {spectral gap ln k}. Then, for each fixed $l$ and $A$, the term $\sup_{1\leq k\leq N} I_1$ vanishes as $N\uparrow \infty$.  

Now, we turn to estimate the term $I_2$. By equivalence of ensembles (cf.\,p.355, \cite{KL}),  
the absolute value in $I_2$ vanishes as $l\uparrow \infty$, uniformly in $k$.
Therefore, the term $I_2$ vanishes as soon as we take $N\uparrow\infty$, $l\uparrow \infty$ in order.

To prove the lemma, it now remains to show
\begin{equation*}
\limsup_{l\to \infty} \limsup_{N\to \infty}
\sup_{0\leq k \leq N }
\E_N
\Big[
\Big |
\int_0^T
V_{k,l,A} (s,\eta_s) ds
\Big |
\Big]
=
0.
\end{equation*}

By the entropy inequality 
(cf.\,p.338 \cite{KL}) 
and the assumption $H(\mu^N| \Pam_{N})\leq C_0N$,
 we have, for any $\gamma>0$
\begin{equation} \label{eqn: 1block, after entropy inequality}
\begin{split}
\E_N
\Big[
\Big |
\int_0^T
V_{k,l,A} (s,\eta_s) ds
\Big |
\Big]
\leq 
\dfrac{C_0}{\gamma}
+
\dfrac{1}{\gamma N} \ln \E_{\Pam_{N}}
\Big[
\exp
 \left\{
  \gamma N
  \Big |
\int_0^T
V_{k,l,A} (s,\eta_s) ds
\Big |
  \right\}
\Big].
\end{split}
\end{equation}
By Feynman-Kac formula (cf.\,p.336, \cite{KL}),
\eqref{eqn: 1block, after entropy inequality} is bounded further by
\[
\dfrac{C_0}{\gamma}
+
\dfrac 1 {\gamma N} 
\int_0^T \lambda_{N,l}(s) ds
\]
where $\lambda_{N,l}(s)$ is the largest eigenvalue of 
$N^2 S + \gamma N V_{k,l,A}(s,\eta)$.

Now, fix $s\in [0,T]$ and omit the argument $s$ to simplify notation.
To estimate the eigenvalue $\lambda_{N,l}(s)$,
we make use of the variational formula:
\begin{equation*}
(\gamma N)^{-1} \lambda_{N,l}
=
\sup_f
\left\{
E_{\Pam_{N}}\left[  V_{k,l,A} f  \right]
-\gamma^{-1} N
E_{\Pam_{N}}\left [
\sqrt f (-S \sqrt f)
\right]
\right\}
\end{equation*}
where the supremum is over all $f$ which are densities with respect to $\Pam_{N}$.

For any density $f$, we consider its restriction with respect to configurations sites $\L_{k,l}$, i.e.\,we define $f_{k,l}= E_{\Pam_N }\big[f|\Omega_{k,l}\big]$.
Recall that $\kappa_{k,l}$ is the restriction of $\Pam_N $ to $\L_{k,l}$, and that $S_{k,l}$ is the localized generator.  
Notice that $E_{\Pam_{N}}\left [\sqrt f (-S_{k,l}  \sqrt f) \right]\leq E_{\Pam_{N}}\left [\sqrt f (-S \sqrt f) \right]$.
By convexity of the Dirichlet form, we have
\begin{equation*}
(\gamma N)^{-1} \lambda_{N,l}
\leq
\sup_{f_{k,l}}
\left\{
E_{\kappa_{k,l}}\left [ V_{k,l,A} f_{k,l}  \right ]
-\gamma^{-1} N
E_{\kappa_{k,l}}\left [
\sqrt {f_{k,l}} (-S_{k,l} \sqrt {f_{k,l}})
\right]
\right\}.
\end{equation*}

We now write $f_{k,l}  d\kappa_{k,l}$ with respect to sets $\Omega_{k,l,j}$ of configurations with total particle number $j$ on $\L_{k,l}$:
\begin{equation} \label {beta 0, c klj}
E_{\kappa_{k,l}}\left[  V_{k,l,A} f_{k,l}  \right]
=
\sum_{j\geq 0}
c_{k,l,j}(f) \int V_{k,l,A} f_{k,l,j} d\kappa_{k,l,j},
\end{equation}
where
$ c_{k,l,j}(f) = \int_{\Omega_{k,l,j}} f_{k,l} d\kappa_{k,l}$, and 
$ f_{k,l,j} = c_{k,l,j}(f)^{-1} \kappa_{k,l} \left( \Omega_{k,l,j} \right) f_{k,l}$.  Here, $\sum_{j\geq 0} c_{k,l,j} = 1 $ and $f_{k,l,j}$ is a density with respect to $\kappa_{k,l,j}$.

Then, on $\Omega_{k,l,j}$, we have
$
\dfrac{S_{k,l} \sqrt {f_{k,l}}}{\sqrt {f_{k,l}}}
=
\dfrac{S_{k,l} \sqrt {f_{k,l,j}}}{\sqrt {f_{k,l,j}}}.
$
By \eqref{beta 0, c klj}, we write
\begin{equation*}
E_{\kappa_{k,l}}\left[
\sqrt {f_{k,l}} (-S_{k,l} \sqrt {f_{k,l}})
\right]
=
\sum_{j\geq 0}
c_{k,l,j}(f) 
E_{\kappa_{k,l,j}}\left[
\sqrt {f_{k,l,j}} (-S_{k,l} \sqrt {f_{k,l,j}})
\right ].
\end{equation*}
Then, we have
\begin{equation*}
(\gamma N)^{-1} \lambda_{N,l}
\leq
\sup_{0\leq j\leq A(2l+1)}
\sup_f
\left\{
E_{\kappa_{k,l,j}}\left[ V_{k,l,A} f  \right]
-\gamma^{-1} N 
E_{\kappa_{k,l,j}}\left [
\sqrt f (-S_{k,l} \sqrt f)
\right]
\right\},
\end{equation*}
where the inside supremum is on densities $f$ with respect to $\kappa_{k,l,j}$.

By Lemma \ref {spectral gap ln k}, we have $C_{k,l,j}$ as the inverse spectral gap estimate of $S_{k,l}$. 
Note also $\|V_{k,l,A}\|_{\infty} \leq  C(A,G)$.
Using the Rayleigh estimation (cf.\,p.\,377, \cite{KL}), we have
\begin{equation*}
\begin{split}
&E_{\kappa_{k,l,j}}\left[   V_{k,l,A} f  \right]
-\gamma^{-1} N 
E_{\kappa_{k,l,j}}\left [
\sqrt f (-S_{k,l} \sqrt f)
\right ]\\
\leq&
\dfrac{\gamma N^{-1}}{1-2C(A,G) C_{k,l,j}\,\gamma N^{-1}}
E_{\kappa_{k,l,j}}\left[
  V_{k,l,A}  (-S_{k,l} )^{-1}
  V_{k,l,A}
\right ].
\end{split}
\end{equation*}

As remarked in the beginning of the proof, $E_{\kappa_{k,l,j}}  V_{k,l,A}= 0$.
Observe that the spectral gap estimate of $S_{k,l}$ in Lemma \ref{spectral gap ln k} also implies that $\|S_{k,l}^{-1}\|_2$, the $L^2(\kappa_{k,l,j})$ norm of the operator $S_{k,l}^{-1}$ on 
mean zero functions, is less than or equal to $C_{k,l,j}$. 
Thus, by Cauchy-Schwarz, we have 
\begin{equation*} 
\begin{split}
E_{\kappa_{k,l,j}}\left[
  V_{k,l,A}  (-S_{k,l} )^{-1}
  V_{k,l,A}
\right ]
\leq
C_{k,l,j}
E_{\kappa_{k,l,j}}\left[
  V_{k,l,A} ^2
\right ].
\end{split}
\end{equation*}

Retracing the steps, we obtain
\begin{equation*}
\begin{split}
\E_N
\Big[
\Big |
\int_0^T
 V_{k,l,A} (\eta_s) ds
\Big |
\Big]
\leq 
\dfrac{C_0}{\gamma}
+
\sup_{0\leq j\leq A(2l+1)}
\dfrac{T\gamma N^{-1} C_{k,l,j}}{1-2C(A,G) C_{k,l,j}\,\gamma N^{-1}}
E_{\kappa_{k,l,j}}\left [
  V_{k,l,A} ^2
\right ]
\end{split}
\end{equation*}
where the last expression vanishes uniformly as $N\to \infty$ 
for $1\leq k \leq N$ and $j\leq A(2l+1)$.
The lemma now is proved by letting $\gamma \to \infty$.
\end{proof}

\section{$2$-block estimate}
In this section we discuss $2$-block estimate which is needed for the replacement lemma.  For brevity, we present the main elements, but omit some proofs, as they are similar to those for the $1$-block estimate.

As for the $1$-block estimate, a spectral gap bound will be needed in the comparison of two `blocks'.
Recall the notation $\L_{k,l}$ from the $1$-block estimate.  For $l\geq 1$ and $|k-k'|>2l$, let  $\L_{k,k',l} = \L_{k,l} \cup \L_{k',l}$.
We introduce the following localized generator $S_{k,k',l}$ governing the coordinates $\Omega_{k,k',l} = \N_0^{\L_{k,k',l}}$.
Inside each block, the process moves as before, but we add an extra bond interaction between sites $k+l$ and $k'-l$:
\begin{align*}
S_{k,k',l}f(\eta)= S_{k,l}f(\eta) + & S_{k',l}f(\eta)
+
\dfrac12  \, g(\eta(k+l))
\mf p_{k+l,k'-l}^N 
\big(f(\eta^{k+l,k'-l}) - f(\eta) \big)\\
&+
\dfrac12 \, g(\eta(k'-l)) \, \mf p_{k'-l,k+l}^N
\big(f(\eta^{k'-l,k+l}) - f(\eta) \big)
\end{align*}
where
\[
\mf p_{k+l,k'-l}^N=
 \dfrac12 + \dfrac{q_{k+l}^N}{\sqrt N}
+
\dfrac{\phi_{k'-l,N}}{\phi_{k+l,N}}
\big( \dfrac12 - \dfrac{q_{k'-l}^N}{\sqrt N}\big),\quad
\mf p_{k'-l,k+l}^N
=
\dfrac12 - \dfrac{q_{k'-l}^N}{\sqrt N}
+
\dfrac{\phi_{k+l,N}}{\phi_{k'-l,N}}
\big( \dfrac12 + \dfrac{q_{k+l}^N}{\sqrt N}\big).
\]
As before, the localized measure $\kappa_{k,k',l}$ defined by $\kappa=\Pam_N$ limited to sites in $\L_{k,k',l}$, 
as well as the canonical measure $\kappa_{k,k',l,j}$ on $\Omega_{k,k',l,j}: = \{\eta\in \Omega_{k,k',l}: \sum_{x\in \L_{k,k',l}} \eta(x) = j\}$, that is $\kappa_{k,k',l}$ is conditioned so that there are exactly $j$ particles counted in $\Omega_{k,k',l}$, are both invariant and reversible for the dynamics.

The corresponding Dirichlet form, with measure $\tilde\kappa$ given by $\kappa_{k,k',l}$ or $\kappa_{k,k',l,j}$, is given by
\begin{equation*}
\begin{split}
E_{\tilde \kappa} \left[ f(-S_{k,k',l} f) \right ]
 =&
\dfrac12  \, \sum_{x,x+1\in \L_{k,k',l}}
 E_{\tilde\kappa} \Big[
g(\eta(k)) \, \mf p_{k,+}^N
\big(f(\eta^{x,x+1}) - f(\eta) \big)^2
\Big]\\
 &+
 \dfrac12  E_{\tilde\kappa} \Big[
g(\eta(k+l))
\mf p_{k+l,k'-l}^N
\big(f(\eta^{k+l,k'-l}) - f(\eta) \big)^2
 \Big].
\end{split}
\end{equation*}

Recall also the generator of symmetric zero-range $L_l$ with respect to $\Lambda_{k,l}$ (cf.~below \eqref{d form asymm}).
Let $L'_l$ be the same generator with respect to $\Lambda_{k',l}$.    Define the generator $L_{l,l}$ with respect to $\L_{l,l}$ by
\begin{align*}
L_{l,l}f(\eta) &= L_l f(\eta) + L'_l f(\eta)
+ \frac{1}{2}\left[f\big(\eta^{k+l, k'-l}\big) - f(\eta)\right]g(\eta(k+l)) \\
&\qquad\qquad\qquad\qquad+\frac{1}{2}\left[f\big(\eta^{k'-l,k+l}\big) - f(\eta)\right]g(\eta(k'-l)).
\end{align*}
When $|k-k'|$ is large, the process governed by $L_{l,l}$ in effect treats the blocks as adjacent, with a connecting bond.

Corresponding to the set-up of the gap bound Lemma \ref{spectral gap ln k},
let $\nu^\rho_{l,l}$ be the product of $4l+2$ distributions with common marginal $\rho$.  One may inspect that $\nu^\rho_{l,l}$ is invariant to the dynamics generated by $L_{l,l}$.
Let now $\nu_{l,l,j}$ be $\nu^\rho_{l,l}$ conditioned on that the total number of particles in the $4l+2$ sites is $j$.
Note that $\nu_{l,l,j}$ is independent of $\rho$.
This canonical measure $\nu_{l,l,j}$ is also invariant to the dynamics.  The corresponding Dirichlet form is given by
\begin{equation*}
\begin{split}
E_{\nu_{l,l,j}} \left[ f(-L_{l,l} f) \right ]
=&\ 
\sum_{x,x+1\in \L_{k,k',l}}
 E_{\nu_{l,l,j}} \Big[ \frac12
g(\eta(x))
 [  f\left(\eta^{x,x+1}  \right)   - f(\eta) ]^2
 \Big]\\
 &+
 E_{\nu_{l,l,j}} \Big [ \frac12
g(\eta(k'-l))
\left [  f\left(\eta^{k'-l,k+l}  \right)   - f(\eta) \right ]^2
 \Big ].
\end{split}
\end{equation*}

Let $ r_{k,k',l,N}^{-1} :=
\min\big\{
\mf p_{k+l,k'-l}^N,
\min_{x,x+1\in \L_{k,k',l}}
\left\{
\mf p_{x,+}^N
\right\}
\big\} $.

\begin{Lem} \label {spectral gap of 2 blocks}
We have the following estimates:
\begin{enumerate}
\item Uniform bound:  For all $\eta\in \Omega_{k,k', l, j}$, we have
\begin{equation*}
\left(\dfrac{\phi_{{\rm min},k,k',l}}{\phi_{{\rm max},k,k',l}} \right)^j
\leq
\dfrac{\kappa_{k,k',l,j} (\eta) }{\nu_{l,l,j}(\eta)}
\leq
\left(\dfrac{\phi_{{\rm max},k,k',l}}{\phi_{{\rm min},k,k',l}} \right)^j
\end{equation*}
 where 
$\dsp \phi_{{\rm min},k,k',l} = \min_{x\in\L_{k,k',l}} \phi_{x,N}$ and $\dsp \phi_{{\rm max},k,k',l} = \max_{x\in\L_{k,k',l}} \phi_{x,N}$.

\item Poincar\'e inequality:
For fixed $j\geq 0$ and $k,k'$ such that $|k-k'|>2l+1$,  we have
\begin{equation*} 
{\rm Var}_{\kappa_{k,k',l,j}}(f)
\leq
C_{k,k',l,j} E_{\kappa_{k,k',l,j}} \big [ f(-L_{k,k',l}f) \big ]
\end{equation*}
where
$
C_{k,k',l,j}
=
C
(4l+2)^2  r_{k,k',l,N} 
\left(\dfrac{\phi_{{\rm max},k,k',l}}{\phi_{{\rm min},k,k',l}} \right)^{2j}
$
for a universal constant $C$.

\item 
There exists constant $C_0$ such that
\[
\sup_{k,k',l,N}\dfrac{\phi_{{\rm max},k,k',l}}{\phi_{{\rm min},k,k',l}}  \leq C_0.
\quad
\limsup_{N\uparrow\infty}\sup_{k',k,l}r_{k,k',l,N} \leq C_0.
\]
Hence, for fixed $j$ and $l$, we have
$\dsp \limsup_{\eps \downarrow 0}\limsup_{N\uparrow\infty} 
 \sup_{2l+1\leq |k'-k| \leq \eps  N}
C_{k,k',l,j}<\infty$.
\end{enumerate}
\end{Lem}
We may repeat proof of Lemma \ref{spectral gap ln k}, step by step, to show Lemma \ref{spectral gap of 2 blocks}; to be brief, we omit the details.



We now state 
a $2$-blocks estimate.

\begin{Lem}[$2$-block estimate] \label{lem: 2 blocks}
 We have
 \begin{equation} \label {eqn: 2 blocks estimate}
\limsup_{l\to \infty}
\limsup_{\eps\to 0} 
\limsup_{N\to \infty} 
 \E_N
\Big[
\dfrac1N \sum_{1\leq k \leq N }
\int_0^T 
\Big|
\Phi\left(\eta_{s}^{l}(k)\right)
-
\Phi\left(\eta_{s}^{\eps N}(k)\right)
\Big|
ds
\Big]
=0.
\end{equation}
\end{Lem}


  \begin{proof}
 As $\Phi(\cdot)$ is Lipschitz, it suffices to show \eqref{eqn: 2 blocks estimate} with the absolute value replaced by $\left |\eta_s^{\theta  N}(k) - \eta_s^l(k) \right |$. We will further approximate  $\eta_s^{\theta  N}(k)$ by $\dfrac{1}{2\theta N+1} \sum_{k'}\eta_s^{l}(k')$ where the summation is over $k'$-s such that $2l+1\leq |k'-k|\leq \theta N$.
By the entropy inequality, the resulting error vanishes as $N\uparrow \infty$.
For each pair $(k, k')$, we view $\L_{k,k',l}$, the union of the two blocks $\L_{k,l}$ and $ \L_{k',l}$, as a single block. 
Let $\eta_s^{l} (k,k')$ denote the average number of particles per site over $\L_{k,k',l}$. After a cutoff of large densities (cf.\,p.~92, \cite{KL}), to prove the lemma, it is enough to show, for any $A>0$, that
\begin{equation*}
\limsup_{l\to \infty}
\limsup_{\theta \to 0}
 \limsup_{N\to \infty}
  \sup_{ 2l+1\leq |k'-k| \leq \theta N }
\E_N
\Big[
\int_0^T 
\Big |
\eta_s^{l}(k')
-
\eta_s^l(k)
 \Big |
 \id_{\{\eta_s^{l} (k,k') \leq A \}}
ds
\Big]
=
0.
\end{equation*}
Let $V_{k,k',l,A}(\eta) := \left| \eta^l (k) -  \eta^l(k') \right| \id_{\{\eta^l(k,k')\leq A\}}$.
Following the proof of Lemma \ref{lem: 1 block}, for fixed $l,\theta,N,k,k'$, in order to estimate
$\E_N \Big[ \int_0^T  V_{k,k',l,A} (\eta_s) ds \Big]$, it suffices to bound
\[
(\gamma N)^{-1} \lambda_{N,l}
=
\sup_f
\left\{
E_{\Pam_{N}} \left [  V_{k,k',l,A} f  \right ]
-\gamma^{-1} N 
E_{\Pam_{N}} \left [
\sqrt f (-S \sqrt f)
\right ]
\right\}
\]
where the supremum is over all $f$ which are densities with respect to $\Pam_{N}$. To make use of spectral gap estimates established in Lemma \ref{spectral gap of 2 blocks}, we will need to compare $E_{\Pam_{N}} \left [
\sqrt f (-S \sqrt f) \right ]$ with $E_{\Pam_{N}} \left [ \sqrt f (-S_{k,k',l} \sqrt f) \right ]$, that is the Dirichlet form corresponding to $S$ restricted on $\L_{k,k',l}$ with an extra bond connecting $\L_{k,l}$ and $\L_{k',l}$. By a careful examination, we have, for some constant $C$, that
$
E_{\Pam_{N}} \left [
\sqrt f (-S_{k,k',l}\sqrt f)
\right ]
\leq
C (1+\theta N) 
E_{\Pam_{N}} \left [
\sqrt f (-S \sqrt f)
\right ].
$

We may now resume the scheme of Lemma \ref{lem: 1 block} and proceed until the end. Notice that for a successful application of Rayleigh estimate, we need to have that $E_{\kappa_{k,k',l,j}}\left[  V_{k,k',l,A}   \right ]$ vanishes where $\kappa_{k,k',l,j}$ is $\Pam_N$ conditioned on configurations with exactly $j$ particles over $\L_{k,k',l}$.
 In fact, by Lemma \ref{spectral gap of 2 blocks},
$E_{\kappa_{k,k',l,j}}\left [ V_{k,k',l,A}   \right ]
\leq {C_0}^j E_{\nu_{l,l,j}}\left [  V_{k,k',l,A}   \right ]$.
As $\nu_\rho$ is product measure with a common marginal independent of $N$, the term $E_{\nu_{l,l,j}}\left [  V_{k,k',l,A}   \right ]$ does not depend on $N$ or $\theta$.   By adding and subtracting $\rho_{j,l}:= j/(2(2l+1))$, we only need to bound $E_{\nu_{l,l,j}}\left[\big |\eta^l(k)-\rho_{j,l} \big|\right]$.  By an equivalence of ensemble estimate (cf.~p.~355 \cite{KL}), we have $E_{\nu_{l,l,j}}\left [  \big|\eta^l(k) -\rho_{j,l} \big| ^2  \right ] \leq C(A){\rm Var}_{\nu^{\rho_{j,l}}_{l,l}}\left( \eta^l(k)\right)$ (recall $\nu^\rho_{l,l}$ defined before Lemma \ref{spectral gap of 2 blocks}). 
Note that the variance is of order $O(l^{-1})$ as the single site variance ${\rm Var}_{\nu^{\rho_{j,l}}_{l,l}}\left(\eta(k)\right)$ is uniformly bounded for $\rho_{j,l} \leq A$. Hence, $E_{\nu_{l,l,j}}\left [  V_{k,k',l,A}   \right ]$ is of order $O(l^{-1/2})$, finishing the proof.  
\end{proof}


\section{Tightness of limit measures} \label {section: tightness}
In this section, we obtain tightness of the family of probability measures $\left\{Q^N\right\}_{N\in \N}$ on the trajectory space $D([0,T],\Mb )$.
We show that $\{Q^N\}$ is tight with respect to the uniform topology, stronger than the Skorokhod topology on $D([0,T], \Mb)$.

\begin{Lem} \label {lem: tightness}
$\left\{Q^N\right\}_{N\in \N}$ is relatively compact with respect to the uniform topology.  As a consequence, all limit points $Q$ are supported on weakly continuous trajectories $\pi$, that is for $G\in C^\infty(\T)$ we have $t\in [0,T]\mapsto\langle G, \pi_t\rangle$ is continuous.
\end{Lem}
\begin{proof}
To deduce that $\{Q^N\}$ is relatively compact with respect to uniform topology, 
we show the following items (cf.\,p.\,51 \cite{KL}).
\begin{enumerate} 
\item For each $t\in[0,T]$, $\epp>0$, there exists a compact set $K_{t,\epp}\subset \Mb$ such that
\begin{equation}\label {eqn1 compactness}
\sup_{N} Q^{N} \left[\pi^N_\cdot: \pi^N_t\notin K_{t,\epp}\right ] \leq \epp.
\end{equation}
\item For every $\epp>0$, 
\begin{equation}\label {eqn2 compactness}
\lim_{\gamma \to 0} 
\limsup_{N\to \infty}
Q^N \Big [  \pi^N_\cdot : \sup_{|t-s|<\gamma} 
d(\pi^N_t,\pi^N_s ) >\epp \Big]
=0.
\end{equation}
 \end{enumerate}

We now consider \eqref{eqn1 compactness}.
Notice that, for any $A>0$, the set $\left\{ \mu\in \Mb: \langle 1,\mu\rangle \leq A  \right\}$ is compact in $\Mb$.
Since the total number of particles is conserved, we have
$
Q^N\big[\langle 1, \pi^N_t\rangle > A\big]
=
Q^N\big[\langle 1, \pi^N_0\rangle > A\big]
\leq
\dfrac1A \E_N \big [ \dfrac1N \sum_{k=1}^N \eta_0(k) \big ]
$.
By \eqref{expected_number}, we have $\E_N \Big [ \dfrac1N \sum_{k=1}^N \eta_0(k) \Big ]\leq C$ for some constant $C<\infty$ independent of $N$ and $A$. Then, the first condition \eqref{eqn1 compactness} is checked by taking $A$ large. 

To show the second condition \eqref {eqn2 compactness}, it is enough to show a counterpart of the condition for the distributions of $\langle G, \pi^N_{\cdot}\rangle$ where $G$ is any smooth test function on $\T$
(cf.\,p.\,54, \cite{KL}).
In other words, we need to show,
 for every $\epp>0$, 
\begin{equation}\label {eqn2 compactness with G}
\lim_{\gamma \to 0} 
\lim_{N\to \infty}
Q^N \Big [  \pi^N_\cdot : \sup_{|t-s|<\gamma} \Big| \langle G,  \pi^N_t \rangle  -\langle G, \pi^N_s \rangle\Big| >\epp \Big]
=0.
\end{equation}
To this end, notice that
$\left\langle G,  \pi^N_t \right\rangle
=
\left\langle G,  \pi^N_0 \right\rangle
+
\int_0^t N^2L \left\langle G,   \pi^N_s \right\rangle ds
+
M^{N,G}_t$,
then we only need to consider the oscillations of $\int_0^t N^2L \left\langle G,   \pi^N_s \right\rangle ds$
and $M^{N,G}_t$ respectively. 

Recall the generator computation \eqref{gen_comp} and the notation $D^{G,s}_{N,k}$ in \eqref{DG}.
As $g$ grows at most linearly and $D^{G,s}_{N,k}$ is bounded (cf.~\eqref{DG_bound}), we have
\begin{equation*}
\begin{split}
&\sup_{|t-s| < \gamma}
 \Big|
 \dsp\int_s^t N^2L \left\langle G,  \pi^N_{\tau} \right\rangle d\tau
     \Big|
 =
 \sup_{|t-s| < \gamma}
 \Big|
 \dsp\int_s^t 
\dfrac1N \sum_{1\leq k \leq N}  D^{G,s}_{N,k} g(\eta_\tau(k)) d\tau
     \Big|\\
 &\leq 
g^*  C_G 
 \sup_{|t-s| < \gamma}
 \dsp\int_s^t 
\Big\{
\dfrac1N \sum_{1\leq k \leq N}
\eta_{\tau}(k)
\Big\}  d\tau = g^* C_G
\gamma 
\dfrac1N \sum_{1\leq k \leq N} \eta_0(k).
\end{split}
\end{equation*}
Recall that $\E_N \Big [ \dfrac1N \sum_{k=1}^N \eta_0(k) \Big ]$ 
is uniformly bounded in $N$.
Then, by Markov inequality, we conclude that
$
Q^N \Big [ 
\sup_{|t-s|<\gamma} 
 \left|
 \int_s^t N^2L \left\langle G,  \pi^N_{\tau} \right\rangle d\tau
     \right|
>\epp \Big]
$
vanishes as $N\uparrow\infty$ and $\gamma\downarrow 0$.

We turn to the martingale $M^{N,G}_t$.
By $\big| M^{N,G}_t - M^{N,G}_s  \big| \leq  \big| M^{N,G}_t \big| + \big| M^{N,G}_s  \big|$, we have
$
\P_N
\big[
\sup_{|t-s|<\gamma} 
\big| 
M^{N,G}_t - M^{N,G}_s  
\big|
>\epp
\big]
\leq
2\P_N
\big[
\sup_{0\leq t \leq T} 
\big| 
M^{N,G}_t  
\big|
>\epp /2
\big]
$.
Using Chebyshev and Doob's inequality, we further bound it by
$$
\dfrac{8}{\epp^2}
\E_N
\Big[
\Big(
\sup_{0\leq t \leq T} 
\big| 
M^{N,G}_t  
\big|
\Big)^2
\Big]
\leq
\dfrac{32}{\epp^2}
\E_N
\Big[
\big(
M^{N,G}_T
\big)^2
\Big]
= 
\dfrac{32}{\epp^2}\E_N
\langle
M^{N,G}
\rangle_T .
$$

By Lemma \ref{lem: martingale bounds}, $\E_N \langle M^{N,G}\rangle_T = O(N^{-1})$.
Then, we conclude
\begin{equation*}
\begin{split}
\lim_{\gamma \to 0}
\lim_{N\to \infty}
\P_N
\Big[
\sup_{|t-s|<\gamma} 
\left| 
M^{N,G}_t - M^{N,G}_s  
\right|
>\epp
\Big]
=0.
\qedhere
\end{split}
\end{equation*}
\end{proof}

\section{Properties of limit measures.}

By Lemma \ref{lem: tightness}, the sequence $\left\{ Q^N\right\}$ is relatively compact with respect to the uniform topology.  Consider any convergent subsequence of $Q^N$ and relabel so that $Q^N \Rightarrow Q$.  
We now consider absolute continuity and an energy estimate for trajectories under $Q$.

\subsection{Absolute continuity}
We now address absolute continuity and conservation of mass propoerties under $Q$.

\begin{Lem}\label{lem: rho < phi c}
$Q$ is supported on absolutely continuous trajectories:
\begin{equation*} 
Q \left[ \pi_\cdot :
\pi_t(dx) = \rho(t,x)dx 
 \text{ for all } 0\leq t\leq T
\right]
=
1.
\end{equation*}
Moreover, for all $0\leq t\leq T$
  we have
$\int_\T \rho(t,x)dx = \int_\T \rho(0,x)dx$.
 \end{Lem}
 
\begin{proof} A standard proof, namely that of Lemma 1.6, p.~73, \cite{KL}, shows the first statement.  The second follows directly from the weak convergence of $Q^N$ to $Q$ and the conservation of mass $\sum_{x\in \T_N}\eta_t(x) = \sum_{x\in \T_N}\eta_0(x)$.
\end{proof}


\subsection{Energy estimate}
We now state an important `energy estimate' for the paths on which $Q$ is supported.  We follow the framework presented in Section 5.7 of \cite{KL}, however, there are major differences due to the inhomogeneous random environment.  Previous bounds on the random environment developed in Section \ref{subsec: invariant measure} will be useful in the argument.

\begin{Prop} \label {prop: weak derivative of Phi}
$Q$ is supported on paths $\rho(t,x)dx$ with the property that
there exists an $L^1([0,T]\times \T)$ function denoted by $\partial_x \Phi(\rho(s,x))$ such that
\begin{equation*}
\begin{split}
\int_0^T \int_\T 
\partial_x G(s,x) \Phi(\rho(s,x)) dx ds
=
-\int_0^T \int_\T 
G(s,x) \partial_x\Phi(\rho(s,x)) dx ds
\end{split}
\end{equation*}
for all $G$ smooth on $[0,T] \times \T$.
\end{Prop}

A main ingredient for the proof of Proposition \ref{prop: weak derivative of Phi} is the following lemma.
For $\epsilon>0$, $\delta>0$, $H(\cdot) \in C^1(\T)$ and $N\in \N$, we define
\begin{equation*}
\begin{split}
&W_N(\epsilon, \delta, H, \eta)
:=
\sum_{1\leq x \leq N} \dfrac{H(x/N)}{\epsilon N}
\left[ \Phi\left(\eta^{\delta N}(x)\right) - \Phi\left(\eta^{\delta N}(x+\epsilon N)\right) \right] \\
&-
\dfrac{4}{c^2 N} 
\sum_{1\leq x \leq N} \dfrac{H^2(x/N)}{\epsilon N} 
\sum_{0 \leq k \leq \epsilon N} \Phi\left(\eta^{\delta N}(x+k)\right)
-
\sum_{1\leq x \leq N} \dfrac{CH(x/N)}{cN} \Phi\left(\eta^{\delta N}(x)\right) .
\end{split}
\end{equation*}
Here, the constants $c$ and $C$, as we recall from Lemma \ref{lem: uniform bounds on phi max min}, come from the inequalities
\[
0<c\leq \min_k \phi_{k,N} \leq \max_k \phi_{k,N}\leq 1, \ \ {\rm and \ \ } \max_k |\phi_{k,N} - \phi_{k+1,N}|\leq \frac{C}{N}.
\]

\begin{Lem} \label {lem: W_N bound}
Let $\{H_j \}_{j\in \N}$ be a dense sequence in $C^{0,1}([0,T]\times \T)$. 
Then, there exists constant $K_0$ such that  for any $m\geq 1$, and $\epsilon >0$, 
\begin{equation*}
\limsup_{\delta\to 0} \limsup_{N\to \infty} 
\E_N 
\Big[
\max_{1\leq j\leq m} 
\Big\{ \int_0^T W_N(\epsilon, \delta, H_j(s,\cdot), \eta_s) ds  \Big\}
\Big]
\leq K_0.
\end{equation*}
\end{Lem}

Before going to the proof of the lemma, we turn to Proposition
\ref{prop: weak derivative of Phi}.

\medskip
\noindent {\it Proof of Proposition \ref{prop: weak derivative of Phi}.}
It follows from Lemma \ref{lem: W_N bound} that 
\begin{equation*}
\begin{split}
&E_Q 
\Big[
\sup_{H\in C^{0,1}([0,T]\times \T)}
\Big\{  \int_0^T  \int_\T \partial_x H(s,x) \Phi(\rho(s,x)) dx  ds\\
&\quad -
 \dfrac{4}{c^2}\int_0^T  \int_\T H^2(s,x)  \Phi(\rho(s,x))    dx  ds
 -
 \dfrac{C}{c}\int_0^T  \int_\T H(s,x)  \Phi(\rho(s,x))    dx  ds
\Big\}
\Big]
\leq K_0,
\end{split}
\end{equation*}
cf.\,p.\,103, Lemma 7.2 in \cite{KL}.
As a result, for $Q$-a.e.\,path $\rho(s,u)du$, there exists $B=B(\rho)$ such that, for all $H\in C^{0,1}([0,T]\times \T)$,
\begin{equation*}
\begin{split}
\int_0^T  \int_\T \partial_x H(s,x) \Phi(\rho(s,x)) dx  ds
-&
 \dfrac{4}{c^2}\int_0^T  \int_\T H^2(s,x)  \Phi(\rho(s,x))    dx  ds\\
 &\quad-
 \dfrac{C}{c}\int_0^T  \int_\T H(s,x)  \Phi(\rho(s,x))    dx  ds
\leq B,
\end{split}
\end{equation*}
Notice that
\begin{equation*}
\begin{split}
\int_0^T  \int_\T H(s,x)  \Phi(\rho(s,x)) dxds
\leq&
\dfrac12\int_0^T  \int_\T H^2 (s,x) \Phi(\rho(s,x))    dx  ds
+
\dfrac12 \int_0^T  \int_\T  \Phi(\rho(s,x))    dx  ds\\
\leq&
\dfrac12\int_0^T  \int_\T H^2 (s,x)  \Phi(\rho(s,x))    dx  ds
+
\dfrac {g^*}{2} \int_0^T  \int_\T  \rho(s,x)   dx  ds\\
=&
\dfrac12\int_0^T  \int_\T H^2 (s,x)  \Phi(\rho(s,x))    dx  ds
+
\dfrac {g^* T }{2}\int_\T \rho(0,x)dx.
\end{split}
\end{equation*}
We obtain
\begin{equation*}
\begin{split}
\int_0^T  \int_\T \partial_x H(s,x) \Phi(\rho(s,x)) dx  ds
-
C^\prime \int_0^T  \int_\T H^2(s,x)  \Phi(\rho(s,x))    dx  ds
\leq
B^\prime.
\end{split}
\end{equation*}
where $C^\prime = \dfrac{4}{c^2} + \dfrac{C}{2c}$ and $B^\prime = \dfrac{Cg^* T\int_\T\rho(0,x)dx}{2c} + B$.
Now, the proof follows exactly from proof of Theorem 7.1, p.\,105, \cite{KL}. \qed

We now return to the proof of Lemma \ref{lem: W_N bound}.

\medskip
\noindent {\it Proof of Lemma \ref{lem: W_N bound}.}
By the replacement lemma (Lemma \ref{replacement_lemma}, and notice that $D_{N,k}^{G,t}$ can be replaced by any bounded function), it suffices to show that there exists constant $K_0$ such that for any $m\geq 1$ and $\epsilon >0$
\begin{equation} \label{tilde W limit}
\limsup_{N\to \infty} 
\E_N 
\Big[
\max_{1\leq j\leq m} 
\Big\{ \int_0^T \tilde W_N(\epsilon,H_j(s,\cdot), \eta_s) ds  \Big\}
\Big]
\leq K_0
\end{equation}
where
\begin{equation*}
\begin{split}
&\tilde W_N(\epsilon, H(\cdot), \eta)
:=
\sum_{1\leq x \leq N} \dfrac{H(x/N)}{\epsilon N}
\left( g(\eta(x)) - g(\eta(x+\epsilon N))   \right) \\
&-
\dfrac{4}{c^2 N} \sum_{1\leq x \leq N} \dfrac{H^2(x/N)}{\epsilon N} \sum_{0\leq k\leq \epsilon N} g(\eta(x+k))
-
\sum_{1\leq x \leq N} \dfrac{CH(x/N)}{cN} g(\eta(x)) .
\end{split}
\end{equation*}

Let $\id_{A,N}(\eta):=\id_{\sum_{k=1}^N \eta(k) \leq AN}$. Define
$W_{A,N}(\epsilon, H(\cdot), \eta) := \tilde W_N(\epsilon, H(\cdot), \eta)  \id_{A,N}(\eta)$.

As in the beginning of proof of the $1$-block estimate, stated in
Lemma \ref{lem: 1 block}, where we cut off high densities, assertion
\eqref{tilde W limit} holds provided we prove that
\begin{equation} \label{tilde W limit, with truncation}
\limsup_{N\to \infty} 
\E_N 
\Big[
\max_{1\leq j\leq m} 
\Big\{ \int_0^T W_{A,N}(\epsilon,H_j(s,\cdot), \eta_s) ds  \Big\}
\Big]
\leq K_0.
\end{equation}
To this end,
by the entropy inequality, the expectation in \eqref{tilde W limit, with truncation} is bounded from above by
\begin{equation*}
\begin{split}
\dfrac1N H(\mu^N| \Pam_{N})
+
\dfrac1N \ln \E_{\Pam_N}
\Big[
\exp
 \Big\{
\max_{1\leq j\leq m} 
\Big\{ N \int_0^T W_{A,N}(\epsilon,H_j(s,\cdot), \eta_s) ds  \Big\}
  \Big\}
\Big].
\end{split}
\end{equation*}
Since the relative entropy $H(\mu^N| \Pam_{N})\leq C_0N$, we obtain the left hand side of  \eqref{tilde W limit, with truncation} is bounded from above by
\begin{equation*}
\begin{split}
C_0
+
\max_{1\leq j\leq m} 
\limsup_{N\to\infty}
\dfrac1N \ln \E_{\Pam_N}
\Big[
\exp
\Big\{
 N \int_0^T W_{A,N}(\epsilon,H_j(s,\cdot), \eta_s) ds
  \Big\}
\Big].
\end{split}
\end{equation*}
By Feynman-Kac formula, for any fixed index $j$, the limsup term in previous expression is bounded from above by
\begin{equation*}
\begin{split}
\limsup_{N\to \infty}
 \int_0^T 
 \sup_f
 \left\{
 E_{\Pam_N} \left[ W_{A,N}(\epsilon, H_j(s,\cdot), \eta) f(\eta) \right]
- N E_{\Pam_N} \Big[\sqrt f (-S\sqrt f)\Big] 
\right\}
ds
\end{split}
\end{equation*}
where the supremum is over all $f$ which are densities with respect to
$\Pam_{N}$. As $c\leq \min_k\phi_{k,N}$ and, by Lemma \ref{lem: qx
  bound}, $q_k^N$ is bounded, the Dirichlet form
$E_{\Pam_N} \Big[\sqrt f (-S\sqrt f)\Big]$ is estimated as
\begin{align} \label {eqn: H_s derivative estimate, iii}
&\sum_{1\leq x \leq N}
E_{\Pam_N}
\Big[
\dfrac12 \Big(  \phi_{x,N}
\big( \dfrac12 + \dfrac{q_x^N}{\sqrt N}\big)
+
\phi_{x+1,N}
\big( \dfrac12 - \dfrac{q_{x+1}^N}{\sqrt N}\big)
\Big)
\big(\sqrt{f(\eta+\delta_{x})} - \sqrt{f(\eta+\delta_{x+1})}\big)^2\Big]\nonumber\\
\geq&
\sum_{1\leq x \leq N}
E_{\Pam_N}
\Big[
\dfrac c 4 
\big(\sqrt{f(\eta+\delta_{x})} - \sqrt{f(\eta+\delta_{x+1})}\big)^2\Big].
\end{align}
Here, we used, for each $x$, 
$E_{\Pam_N} [g(\eta(x)) f(\eta)]
=
E_{\Pam_N} [\phi_{x,N} f(\eta+\delta_x)]$, where $\delta_x$ stands for the configuration with the only particle at $x$; $\eta + \delta_x$ is the configuration obtaining from adding one particle at $x$ to $\eta$.

It now remains to show, for all $H$ in $C^{0,1}([0,T]\times \T)$, that
\begin{equation} \label{eqn: log exp negative}
\begin{split}
 E_{\Pam_N} \left[ W_{A,N}(\epsilon, H(s,\cdot), \eta) f(\eta) \right]
- N E_{\Pam_N} \Big[\sqrt f (-S\sqrt f)\Big] 
\leq 0.
\end{split}
\end{equation}
We first compute that $E_{\Pam_N} \left[ W_{A,N}(\epsilon, H(s,\cdot), \eta) f(\eta) \right]$ equals
\begin{equation} \label {eqn: H_s derivative estimate}
\begin{split}
&
E_{\Pam_N}
\Big[
\sum_{1\leq x \leq N} \dfrac{H(x/N)}{\epsilon N}
( g(\eta(x)) - g(\eta(x+\epsilon N)) \id_{A,N}(\eta) f(\eta)
\Big]\\
&-
\dfrac{4}{c^2 N} 
E_{\Pam_N}
\Big[ \sum_{1\leq x \leq N} \dfrac{H^2(x/N)}{\epsilon N} \sum_{0\leq k\leq \epsilon N} g(\eta(x+k)) \id_{A,N}(\eta)
\Big]\\
&-
E_{\Pam_N}
\Big[ 
\sum_{1\leq x \leq N} \dfrac{CH(x/N)}{cN} g(\eta(x)) \id_{A,N}(\eta)
\Big].
\end{split}
\end{equation}

Notice that
$E_{\Pam_N} [g(\eta(x)) \id_{A,N}(\eta) f(\eta)]$ may be written as
\begin{align*}
E_{\Pam_N} [\phi_{x,N} \id_{A,N}(\eta+\delta_x) f(\eta+\delta_x)]
&=
E_{\Pam_N} [\phi_{x,N} \id_{A-1/N,N}(\eta) f(\eta+\delta_x)].
\end{align*}
 Then, the first expectation in 
  \eqref{eqn: H_s derivative estimate} is written as
\begin{equation} \label {eqn: H_s derivative estimate, rewritten}
\begin{split}
&\sum_{x=1}^N \dfrac{H(x/N)}{\epsilon N}
E_{\Pam_N}
\Big[
 \big(\phi_{x,N} f(\eta+\delta_x) - \phi_{x+\epsilon N,N} f(\eta+\delta_{x+\epsilon N})
\big) \id_{A-1/N,N}(\eta)
\Big]\\
\leq&
\sum_{x=1}^N \dfrac{H(x/N)}{\epsilon N}
E_{\Pam_N}
\Big[
\big(
f(\eta+\delta_x)( \phi_{x,N}  - \phi_{x+\epsilon N,N} )
\big)
\id_{A-1/N,N}(\eta)
\Big]\\
&\quad+
\sum_{x=1}^N \dfrac{H(x/N)}{\epsilon N}
E_{\Pam_N}
\Big[
\big(
\phi_{x+\epsilon N,N} (f(\eta+\delta_x) -  f(\eta+\delta_{x+\epsilon N}) )
\big)
\id_{A-1/N,N}(\eta)
\Big].
\end{split}
\end{equation}
Using that $ 0<c\leq \min_{1\leq k\leq N} \phi_{k,N} \leq \max_k\phi_{k,N} \leq 1$ and 
$ \max_{1\leq k\leq N} |\phi_{k,N} - \phi_{k+1,N}| \leq C N^{-1}$,
the first sum in the right hand side of the inequality \eqref {eqn: H_s derivative estimate, rewritten}
is bounded from above by
\begin{align}  \label{eqn: H_s derivative estimate, i}
&\sum_{x=1}^N \dfrac{H(x/N)}{N}
E_{\Pam_N}
\Big[
f(\eta+\delta_x)\dfrac{|\phi_{x,N}  - \phi_{x+\epsilon N,N} |}{\epsilon}
\id_{A-1/N,N}(\eta)
\Big]\\
\leq&
\sum_{x=1}^N \dfrac{CH(x/N)}{N}
E_{\Pam_N}
\big[
f(\eta+\delta_x) \id_{A-1/N,N}(\eta)
\big]\nonumber\\
=&
\sum_{x=1}^N \dfrac{CH(x/N)}{N\phi_{x,N}}
E_{\Pam_N}
\big[
g(\eta(x))  \id_{A,N}(\eta) f(\eta)
\big]
\ \leq \ 
\sum_{x=1}^N \dfrac{CH(x/N)}{cN}
E_{\Pam_N}
\big[
g(\eta(x)) \id_{A,N}(\eta) f(\eta)
\big]. \nonumber
\end{align}
Now, we proceed to the second sum in the right hand side of  \eqref {eqn: H_s derivative estimate, rewritten}. Using $\displaystyle 0<c\leq \min_k \phi_{k,N} \leq \max_k\phi_{k,N} \leq 1$,
the sum is bounded from above by
\begin{equation*} \label {eqn: H_s derivative estimate, 2nd term}
\begin{split}
\sum_{x=1}^N \dfrac{H(x/N)}{\epsilon N}
E_{\Pam_N}
\Big[
\big(f(\eta+\delta_x) -  f(\eta+\delta_{x+\epsilon N}) \big) \id_{A-1/N,N}(\eta)
\Big]
\end{split}
\end{equation*}
which is rewritten as
\begin{equation} \label{eqn: H_s derivative estimate, 2nd term, rewritten}
\begin{split}
&
E_{\Pam_N}
\Big[
\sum_{x=1}^N 
\sum_{k=0}^{\epsilon N -1}
\dfrac{H(x/N)}{\epsilon N}
\big(f(\eta+\delta_{x+k}) -  f(\eta+\delta_{x+k+1})\big) \id_{A-1/N,N}(\eta)
\Big]\\
=&
E_{\Pam_N}
\Big[
\sum_{x=1}^N 
\sum_{k=0}^{\epsilon N -1}
\dfrac{H(x/N)}{\epsilon N}
\id_{A-1/N,N}(\eta)
\big(\sqrt{f(\eta+\delta_{x+k})} + \sqrt{f(\eta+\delta_{x+k+1})}\big)\\
&\qquad\qquad\qquad\qquad\qquad\qquad\qquad \qquad\times
\big(\sqrt{f(\eta+\delta_{x+k})} - \sqrt{f(\eta+\delta_{x+k+1})}\big)
\Big].
\end{split}
\end{equation}
Using $2ab\leq a^2 + b^2$, for any $\beta>0$, \eqref{eqn: H_s derivative estimate, 2nd term, rewritten} is bounded from above by
\begin{equation} \label{eqn: H_s estim, separated}
\begin{split}
&E_{\Pam_N}
\Big[
\sum_{x=1}^N 
\sum_{k=0}^{\epsilon N -1}
\dfrac{H^2(x/N)}{2\epsilon N \beta}
\id_{A-1/N,N}(\eta)
\big(\sqrt{f(\eta+\delta_{x+k})} + \sqrt{f(\eta+\delta_{x+k+1})}\big)^2
\Big]\\
&\qquad\qquad\qquad \qquad+
E_{\Pam_N}
\Big[
\sum_{x=1}^N 
\sum_{k=0}^{\epsilon N -1}
\dfrac{\beta}{2\epsilon N}
\big(\sqrt{f(\eta+\delta_{x+k})} - \sqrt{f(\eta+\delta_{x+k+1})}\big)^2
\Big].
\end{split}
\end{equation}
The first expectation in \eqref{eqn: H_s estim, separated} is bounded from above by
\begin{equation} \label {eqn: H_s derivative estimate, ii}
\begin{split}
&E_{\Pam_N}
\Big[
\sum_{x=1}^N 
\sum_{k=0}^{\epsilon N -1}
\dfrac{H^2(x/N)}{2\epsilon N \beta}
2 \big(f(\eta+\delta_{x+k})+ f(\eta+\delta_{x+k+1}) \big)
\id_{A-1/N,N}(\eta)
\Big]\\
=&
\sum_{x=1}^N 
\dfrac{H^2(x/N)}{\epsilon N \beta}
\sum_{k=0}^{\epsilon N -1}
E_{\Pam_N}
\Big[\Big(\dfrac{ g(x+k)}{\phi_{x+k,N}} + \dfrac{ g(x+k+1)}{\phi_{x+k+1,N}}  \Big)
\id_{A,N}(\eta)  f(\eta)
\Big]\\
\leq&
\sum_{x=1}^N 
\dfrac{2 H^2(x/N)}{c\epsilon N \beta}
\sum_{k=0}^{\epsilon N}
E_{\Pam_N}
\Big[ g(x+k) \id_{A,N}(\eta) f(\eta) \Big].
\end{split}
\end{equation}
The second expectation in \eqref{eqn: H_s estim, separated} is rewritten and bounded, noting \eqref{eqn: H_s derivative estimate, iii}, as
\begin{equation*} \label {eqn: A_k - A_k+1, rewritten}
\begin{split}
E_{\Pam_N}
\Big[
\sum_{x=1}^N 
\dfrac{\beta}{2}
\big(\sqrt{f(\eta+\delta_{x})} - \sqrt{f(\eta+\delta_{x+1})}\big)^2
\Big] \leq \frac{2\beta}{c} E_{\Pam_N} \Big[\sqrt f (-S\sqrt f)\Big].
\end{split}
\end{equation*}
Now, we set $\beta = cN/2$.
Putting together \eqref {eqn: H_s derivative estimate}, \eqref{eqn: H_s derivative estimate, i}, and \eqref{eqn: H_s derivative estimate, ii}, 
we obtain  \eqref{eqn: log exp negative}. \qed

\section{Uniqueness of Weak Solutions} \label{section uniqueness} 
In this section, we present results on uniqueness of weak solutions to the PDE
\begin{equation} \label{eqn: main PDE general}
\begin{cases}
\partial_t \rho(t,x)
=
\partial_{xx} \Phi(\rho(t,x))
+
\partial_x \left( \alpha(x) \Phi(\rho(t,x))\right), \quad x\in \T, t\geq 0\\
\rho(0,x) = \rho_0(x)
\end{cases}
\end{equation}
where $\Phi(\cdot)\in C^1[0,\infty)$ satisfies $0\leq \Phi'(\cdot)\leq C_g$ (cf.~below \eqref{Phi_eqn});  
$|\alpha(x)|<A$ for some constant $A<\infty$. and $\rho_0$ is nonnegative and belongs to the class
$L^1(\T)$.
\begin{Def}
\label{weak_def}
We say $\rho(t,x): [0,T]\times \T \mapsto [0,\infty)$ is a weak solution to \eqref{eqn: main PDE general} if 
\begin{enumerate}
\item 
$\int_\T \rho(t,x) dx = \int_\T \rho_0(x) dx$ for all $t\in[0,T]$.
\item
$\rho(t,\cdot)$ is weakly continuous, that is, for all $G\in C(\T)$, $\int_\T G(x) \rho(t,x) dx$ is a 
continuous function in $t$.
\item 
There exists an $L^1([0,T]\times \T)$ function denoted by $\partial_x \Phi(\rho(s,x))$ such that
for all $G(s,x) \in C^{0,1} \left( [0,T]\times \T  \right)$
\begin{equation} \label {eqn: weak derivative partial x Phi}
\begin{split}
\int_0^T \int_\T 
\partial_x G(s,x) \Phi(\rho(s,x)) dx ds
=
-\int_0^T \int_\T 
G(s,x) \partial_x\Phi(\rho(s,x)) dx ds.
\end{split}
\end{equation}
\item 
For all $G(s,x) \in C_c^\infty \left( [0,T)\times \T  \right)$
\begin{equation*}
\begin{split}
&\int_0^T \int_\T \partial_sG(s,x) \rho(s,x) dx ds
+ \int_\T G(0,x) \rho_0(x) dx\\
&\qquad\qquad=
\int_0^T \int_\T 
\left[
 -\partial_{xx}G(s,x) \Phi(\rho(s,x)) 
+
\partial_x G(s,x) (\alpha(x) \Phi(\rho(s,x)) )
\right]dx ds.
\end{split}
\end{equation*}
\end{enumerate}
\end{Def}

\begin{Prop}
There exists at most one weak solution to \eqref{eqn: main PDE general}.
\end{Prop}
\begin{proof}
Since $\partial_x  \Phi(\rho(s,x)) $ exists,
for all $G(s,x) \in C_c^\infty \left( (0,T)\times \T  \right)$, we have
\begin{equation} \label {eqn: weak sln, with partial x Phi}
\begin{split}
\int_0^T \int_\T \partial_sG(s,x) \rho(s,x) dx ds
=\int_0^T \int_\T
 \partial_x G(s,x) 
\left[
\partial_x \Phi(\rho(s,x))  +  \alpha(x) \Phi(\rho(s,x))
\right]dx ds.
\end{split}
\end{equation}
Define $\varphi(s,x) = \int_0^x \rho(s,u) du$.
For any $G(s,x) \in C_c^\infty \left( (0,T)\times \T  \right)$, let
\begin{equation*}
F(s,x) := \int_0^x G(s,x) dx - x \int_\T G(s,x) dx.
\end{equation*}
Note that $F(s,x)$ is also in the space $C_c^\infty \left( (0,T)\times \T  \right)$.
Therefore we may apply \eqref{eqn: weak sln, with partial x Phi} for $F(s,x)$.
The left hand side, after integration by parts, becomes
\begin{equation} \label{eqn: lhs of using F}
\int_0^T \left[ \partial_s F(s,x) \varphi (s,x)\right] \big\vert_0^1 ds
-
\int_0^T \int_0^1 \partial_s \Big(G(s,x) - \int_0^1 G(s,x)dx\Big) \varphi(s,x) dx ds.
\end{equation}
The term $\int_0^T \left[ \partial_s F(s,x) \varphi (s,x)\right] \big\vert_0^1 ds$ vanishes as the total mass $\int_\T \rho(s,x)dx$ is conserved. Then, \eqref{eqn: lhs of using F} can be rewritten as
\[
\begin{split}
&-
\int_0^T \int_0^1 \partial_s \Big(G(s,x) - \int_0^1 G(s,x)dx\Big) \varphi(s,x) dx ds\\
&\qquad\qquad=
-\int_0^T \int_0^1 \partial_s  G(s,x)  \Big( \varphi(s,x) - \int_0^1 \varphi(s,x) dx\Big)  dx ds.
\end{split}
\]
Define $\phi(s,x) = \varphi(s,x) - \int_0^1 \varphi(s,x) dx$. We now have
\begin{equation} \label {eqn: weak derivative, time w/ G(s,x)}
\begin{split}
\int_0^T \int_0^1 \partial_sG(s,x) \phi(s,x) dx ds.
=
-\int_0^T \int_0^1
G(s,x) h(s,x) dx ds
\end{split}
\end{equation}
where
\begin{equation*}
h(s,x)=
\partial_x \Phi(\rho(s,x))
 +  \alpha(x) \Phi(\rho(s,x))  -  \int_0^1 \alpha(x) \Phi(\rho(s,x)) dx.
\end{equation*}
By straightforward approximation, we obtain from \eqref {eqn: weak derivative, time w/ G(s,x)}, for any $G(\cdot)\in C_c^\infty(0,T)$ and $q(\cdot)\in L^\infty [0,1]$,
\begin{equation} \label {eqn: weak derivative, time}
\begin{split}
\int_0^T  \partial_sG(s) \Big[ \int_0^1 q(x) \phi(s,x) dx\Big] ds.
=
-\int_0^T G(s) \Big[ \int_0^1
 q(x) h(s,x) dx \Big] ds
\end{split}
\end{equation}
As $\phi(s,\cdot)$ and $h(s,\cdot)$ are both in $L^1 \big([0,T]; L^1(\T)\big)$, 
\eqref{eqn: weak derivative, time} implies that 
$\dfrac{d}{ds} \phi(s,\cdot) $, the weak derivative of $\phi(s,\cdot)$,
exists and $\dfrac{d}{ds} \phi(s,\cdot) = h(s,\cdot)$.
Moreover,  in terms of the Bochner integral (cf. \cite{Evans}[p. 302]),
\begin{equation} \label{eqn: C1 phi}
\phi(t,\cdot) = \int_0^t \dfrac{d}{ds} \phi(s,\cdot) ds + \phi(0,\cdot).
\end{equation}

Now, assume there are two solutions $\rho_1,\rho_2$ and therefore
$\phi_1,\phi_2$.  If we show that $\phi_1=\phi_2$, then it follows
$\varphi_1(s,x) - \varphi_2(s,x) = \int_0^1\varphi_1(s,u)
-\varphi_2(s,u)du$ for all $s,x$.  By conservation of mass, it holds
$\varphi_1(s,1) - \varphi_2(s,1)=0$ for all $s$. Then, we conclude
$\varphi_1=\varphi_2$, hence, $\rho_1=\rho_2$ a.e..

 To this end, let $\overline \phi = \phi_1 - \phi_2$
  and $\overline \Phi_{s,x} = \Phi(\rho_1(s,x)) - \Phi(\rho_2(s,x))$.
Therefore, by Lemma \ref{lem: energy equality}, we obtain
\begin{equation}
\label{help1}
\begin{split}
&\dfrac12 \int_0^1 \left( \overline \phi(t,x) \right)^2 dx 
-
\dfrac12 \int_0^1  \left( \overline \phi(0,x) \right)^2 dx \\
=&
\int_0^t \int_0^1
\overline\phi (s,x)
\left[
\partial_x \overline \Phi_{s,x}  +  \alpha(x) \overline \Phi_{s,x} 
\right]dx ds
-
\int_0^t \int_0^1
\overline\phi (s,x)
\left[
\int_0^1 \alpha(u) \overline\Phi_{s,u} du
\right]dx ds.
\end{split}
\end{equation}
Notice, from Lemma \ref{lem: phi partial Phi integrability}, that
\begin{equation*}
\int_0^t \int_0^1
\overline\phi (s,x)
\partial_x \overline \Phi_{s,x}  dx ds
=
-\int_0^t \int_0^1
\partial_x \overline\phi (s,x)
\overline \Phi_{s,x}  dx ds
=
-\int_0^t \int_0^1
 (\partial_x \overline \phi(s,x))^2
 \Phi^\prime_{s,x}  dx ds.
\end{equation*}
Here, we have applied the mean value theorem so that $\overline \Phi_{s,x}=\Phi^\prime_{s,x} \partial_x\overline \phi(s,x)$.

Let $A$ be such that $ |\alpha(x)| \leq A <\infty$. Note that $\Phi(\cdot)$ is an increasing function and $0<\Phi'(\cdot)\leq C_g$.  We have,
\begin{align*}
&\int_0^t \int_0^1
\overline\phi (s,x)
 \alpha(x) \overline \Phi_{s,x} 
dx ds
\ \leq \ 
\int_0^t \int_0^1
A
\left|\overline\phi (s,x)  \right|
\left|\partial_x \overline\phi (s,x) \right|
\Phi^\prime_{s,x} 
dx ds\\
\leq&
\int_0^t \int_0^1
\Big[
\dfrac {A^2}{2} \left( \overline \phi(s,x) \right)^2 \Phi^\prime_{s,x} 
+
\dfrac12
\left(\partial_x \overline\phi (s,x) \right)^2
\Phi^\prime_{s,x} 
\Big] dx ds\\
\leq&
\dfrac {A^2 C_g}2
\int_0^t \int_0^1
\left( \overline \phi(s,x) \right)^2  dx ds
+
\dfrac 12
\int_0^t \int_0^1
\left(\partial_x \overline\phi (s,x) \right)^2
\Phi^\prime_{s,x}  dx ds.
\end{align*}
We also have 
\begin{equation*}
\begin{split}
&- \int_0^t \int_0^1
\overline\phi (s,x)
\left[
\int_0^1 \alpha(u) \overline\Phi_{s,u} du
\right]dx ds\\
\leq&
A\sqrt{C_g} \int_0^t 
\left(\int_0^1
\left|\overline\phi (s,x)\right| dx\right)
 \cdot 
\left(\int_0^1
\left| \sqrt{\Phi^\prime_{s,x}} \partial_x \overline\phi (s,x) 
 \right| dx\right)
  ds\\
\leq&
\dfrac{A^2 C_g}{2}
 \int_0^t 
\left(\int_0^1
\left|\overline\phi (s,x)\right| dx\right)^2 ds
+
\dfrac12
 \int_0^t
\left(\int_0^1
\left|  \sqrt{\Phi^\prime_{s,x}}  \partial_x \overline\phi (s,x) 
 \right| dx\right)^2
ds\\
\leq&
\dfrac{A^2 C_g}{2}
 \int_0^t  \int_0^1
\left( \overline\phi (s,x) \right)^2 dx ds
+
\dfrac12
 \int_0^t
 \int_0^1
\Phi^\prime_{s,x} 
\left(  \partial_x \overline\phi (s,x)   \right)^2 
dx ds.
\end{split}
\end{equation*}
Putting together the above, from equation \eqref{help1}, we get
\begin{equation*}
\int_0^1\left( \overline\phi(t,x)\right)^2 dx 
-
\int_0^1\left( \overline\phi(0,x)\right)^2 dx
\leq
2A^2 C_g
 \int_0^t  \int_0^1
\left( \overline\phi (s,x) \right)^2 dx ds.
\end{equation*}
Notice that  $\overline\phi(0,x) = 0$.
The desired result now follows from Gronwall's inequality.
\end{proof}

\begin{Lem} \label {lem: energy equality}
Let $\overline h(s,\cdot) := \dfrac{d}{ds} \overline \phi(s,\cdot)$. 
We have
\begin{equation*}
\begin{split}
\dfrac12 \int_0^1 \left( \overline \phi(t,x) \right)^2 dx 
-
\dfrac12 \int_0^1  \left( \overline \phi(0,x) \right)^2 dx
=
\int_0^t \int_0^1
\overline\phi (s,x) \overline h(s,x)
dx ds.
\end{split}
\end{equation*}
\end{Lem}
\begin{proof}
For each $n\in \N$, let $t_{k,n} = \dfrac{kt}{n}$, $k=0,1,\ldots,n-1$.
Then
\begin{equation*}
\begin{split}
& \dfrac12 \int_0^1 \Big[\left( \overline \phi(t,x) \right)^2 
 -
 \left( \overline \phi(0,x) \right)^2 \Big]dx
\ = \
  \dfrac12\int_0^1 \sum_{k=0}^{n-1}
 \Big[\left( \overline \phi(t_{k+1,n},x) \right)^2 
 -
 \left( \overline \phi(t_{k,n},x) \right)^2 \Big]dx\\
 =&
  \int_0^1 \sum_{k=0}^{n-1}
  \Big[
   \dfrac{\overline \phi(t_{k+1,n},x) + \overline \phi(t_{k,n},x) }2
 \int_{t_k}^{t_{k+1}} \overline h(s,x)
  ds\Big]
  dx\\
  =&
 \sum_{k=0}^{n-1} \int_{t_k}^{t_{k+1}}
    \int_0^1
   \dfrac{\overline \phi(t_{k+1,n},x) + \overline \phi(t_{k,n},x) }2
\overline h(s,x)
dx  ds
\end{split}
\end{equation*}
Define $\overline \phi_n (s,x) := \dfrac{\overline \phi(t_{k+1,n},x) + \overline \phi(t_{k,n},x) }2$ if $s\in [t_{k,n},t_{k+1,n})$. Then, by weak continuity of $\rho(t,x)$, we have that $ \overline\phi_n (s,x) $ converges a.e.\,to $\overline \phi(s,x)$ on $[0,T]\times \T$. By dominated convergence, noting that $\overline h(s,x)
$ belongs to $L^1([0,T], L^1(\T))$,  
we have, as $n\uparrow\infty$,
\begin{equation} \label {eqn: energy equality, limit form}
\begin{split}
 \dfrac12 \int_0^1 \Big[\left( \overline \phi(t,x) \right)^2 
 -
 \left( \overline \phi(0,x) \right)^2 \Big]dx
 =&
 \int_0^t \int_0^1 \overline \phi_n (s,x) 
\overline h(s,x)
 dx  ds\\
  \to&
   \int_0^t \int_0^1 \overline \phi (s,x) 
\overline h(s,x)
 dx  ds.
 \end{split}
\end{equation}
Since the left hand side of \eqref{eqn: energy equality, limit form} is independent of $n$, the lemma is proved.
\end{proof}

\begin{Lem} \label {lem: phi partial Phi integrability}
We have
$ \dsp
\int_0^t \int_0^1
\overline\phi (s,x)
\partial_x \overline \Phi_{s,x}  dx ds
=
-\int_0^t \int_0^1
\partial_x \overline\phi (s,x)
\overline \Phi_{s,x}  dx ds.
$
\end{Lem}
\begin{proof}
We first extend the class of test functions in
\eqref{eqn: weak derivative partial x Phi}. 
Assume that $F: [0,T] \times [0,1]\mapsto \R$ satisfies the following:
(1) $F$ is measurable; (2) for any fixed $s$, $F(s,\cdot)$ is absolutely continuous; (3) there is a constant $C<\infty$ such that $| \partial_x F(s,x) | \leq C$ for almost all $s,x$; (4) $\int_0^1 \partial_x F(s,x)dx = 0$ for all $s$.

Let $\tau_\epsilon(x)$ be the standard mollifier supported on $[-\epsilon, \epsilon]$.  Define 
\begin{align*}
F_\epsilon(s,x) =\int_0^T \int_\T F(s,u) \tau_\epsilon(x-u) \tau_\epsilon(s-q) dudq
\end{align*}
 with $F$ extended to be $0$ for $s\notin [0,T]$.
By \eqref{eqn: weak derivative partial x Phi},
\[
\int_0^T \int_\T 
F_\epsilon(s,x) \partial_x \overline\Phi_{s,x}   dx ds
=
-
\int_0^T \int_\T 
\partial_x F_\epsilon(s,x)  \overline\Phi_{s,x}  dx ds .
\]
Taking $\epsilon \to 0$, as $\partial_x F$ (and therefore $F$) is bounded and $\overline \Phi_{s,x}$ is integrable, dominated convergence gives
\begin{equation} \label{ean: partial x Phi, extended F}
\int_0^T \int_\T 
F(s,x) \partial_x \overline\Phi_{s,x} dx ds
=
-
\int_0^T \int_\T 
\partial_x F(s,x) \overline\Phi_{s,x} dx ds.
\end{equation}

We now extend the admissible test functions further from $F$ to $\overline \phi$ as claimed in the lemma.
Introduce a truncation on $\partial_x \overline \phi$:
\begin{equation*}
(\partial_x \overline \phi)_{A,s,u} 
=
\begin{cases}
\partial_x \overline \phi (s,u) &-A \leq \partial_x \overline \phi (s,u)\leq A \\
A &\text{otherwise }.
\end{cases}
\end{equation*}
Apply \eqref{ean: partial x Phi, extended F} with 
$F(s,x) =\big[ \int_0^x (\partial_x \overline \phi)_{A ,s,u} du - x \int_0^1 (\partial_x \overline \phi)_{A ,s,u} du\big] \id_{[0,t]}(s)$ to get
\begin{equation*}
\begin{split}
&\int_0^t \int_0^1
\Big[
\int_0^x (\partial_x \overline \phi)_{A,s,u}  du
-
x \int_0^1 (\partial_x \overline \phi)_{A,s,u}  du
\Big]
\partial_x \overline \Phi_{s,x}  dx ds\\
=&
-\int_0^t \int_0^1
\Big[
(\partial_x \overline \phi)_{A,s,x} 
-
\int_0^1 (\partial_x \overline \phi)_{A,s,u} du
\Big]
\overline \Phi_{s,x}  dx ds.
\end{split}
\end{equation*}
As $| (\partial_x \overline \phi)_{A,s,x} | \leq | \partial_x \overline \phi(s,x)|
= |\rho_1(s,u) - \rho_2(s,u)|$, by conservation of mass, and that $\partial_x \overline \Phi_{s,x}$ is integrable, we have by dominated convergence that
\begin{equation*}
\begin{split}
&\lim_{A\to \infty} \int_0^t \int_0^1
\Big[
\int_0^x (\partial_x \overline \phi)_{A,s,u}  du
-
x \int_0^1 (\partial_x \overline \phi)_{A,s,u}  du
\Big]
\partial_x \overline \Phi_{s,x}  dx ds\\
=&
\int_0^t \int_0^1
\Big[
\int_0^x \partial_x \overline \phi (s,u) du
-
x \int_0^1 \partial_x \overline \phi (s,u) du
\Big]
\partial_x \overline \Phi_{s,x}  dx ds
=
\int_0^t \int_0^1 \overline \phi (s,x) \partial_x \overline \Phi_{s,x}  dx ds.
\end{split}
\end{equation*}
Here, we used $ \int_0^1 \partial_x \overline \phi (s,u) du = \int_0^1 (\rho_1(s,u) - \rho_2(s,u))du = 0$.
Similarly,
\begin{equation*}
\begin{split}
\lim_{A\to \infty} \int_0^t \int_0^1
\Big[
\int_0^1 (\partial_x \overline \phi)_{A,s,u} du
\Big]
 \overline \Phi_{s,x}  dx ds
=
0.
\end{split}
\end{equation*}
Finally, notice that 
$(\partial_x \overline \phi)_{A,s,x} \overline \Phi_{s,x}
=(\partial_x \overline \phi)_{A,s,x} \partial_x\overline\phi(s,x) \Phi^\prime_{s,x}$
increases in $A$ since $\Phi'_{s,x}\geq 0$ (cf. \eqref{Phi_eqn}). By monotone convergence, we have
\begin{equation*}
\begin{split}
\lim_{A\to \infty} \int_0^t \int_0^1
(\partial_x \overline \phi)_{A,s,x} 
 \overline \Phi_{s,x}  dx ds
=
 \int_0^t \int_0^1
\partial_x \overline \phi (s,x)
 \overline \Phi_{s,x}  dx ds,
\end{split}
\end{equation*}
finishing the proof.
\end{proof}

\medskip
\noindent{\bf Acknowledgements.} C.L.~has been partly supported by
FAPERJ CNE E-26/201.207/2014, by CNPq Bolsa de Produtividade em
Pesquisa PQ 303538/2014-7, by ANR-15-CE40-0020-01 LSD of the French
National Research Agency.  S.S.~was partly supported by
ARO-W911NF-18-1-0311 and a Simons Foundations Sabbatical grant.


\end{document}